\documentclass[article]{elsarticle}

\usepackage{hyperref}

\usepackage{amsmath,amssymb,enumerate,amsfonts}

\usepackage[normalem]{ulem} 
\newcommand\hl{\bgroup\markoverwith
	{\textcolor[rgb]{0.8,0.8,0.8}{\rule[-.5ex]{2pt}{2.5ex}}}\ULon}
\usepackage{longtable,threeparttable,float}
\usepackage[table]{xcolor}
\usepackage{graphicx, subfigure}
\usepackage{caption}

\biboptions{sort&compress}
\usepackage[numbers,sort&compress]{natbib}

\usepackage[margin=2.5cm]{geometry}	

\usepackage[font=scriptsize]{caption} 

\journal{Journal of \LaTeX\ Templates}

\newenvironment{newlist}[1]%
{\begin{list}{}{\settowidth{\labelwidth}{\bf #1}%
			\setlength{\leftmargin}{\labelwidth}%
			\addtolength{\leftmargin}{\labelsep}%
			}}%
{\end{list}}

\newtheorem{remark}{Remark}[section]
\newtheorem{example}{Example}[section]
\newtheorem{theorem}{Theorem}[section]
\newtheorem{proposition}{Proposition}[section]
\newtheorem{lemma}{Lemma}[section]

\newenvironment{proof}{{\noindent \it Proof.}}{\hfill $\square$\par}

\begin{document}

\begin{frontmatter}

\title{Memory gradient method for multiobjective optimization}

\author[mymainaddress]{Wang Chen}\corref{mycorrespondingauthor}
\cortext[mycorrespondingauthor]{\noindent Corresponding author.}
\ead{chenwangff@163.com}
\author[mysecondaryaddress]{Xinmin Yang}
\ead{xmyang@cqnu.edu.cn}
\author[mythirdaryaddress]{Yong Zhao}
\ead{zhaoyongty@126.com}
\address[mymainaddress]{College of Mathematics, Sichuan University, Chengdu 610065, China}
\address[mysecondaryaddress]{School of Mathematical Sciences, Chongqing Normal University, Chongqing 401331, China}
\address[mythirdaryaddress]{College of Mathematics and Statistics, Chongqing Jiaotong University, Chongqing 400047, China}

\begin{abstract}
In this paper, we propose a new descent method, termed as multiobjective memory gradient method, for finding Pareto critical points of a multiobjective optimization problem. The main thought in this method is to select a combination of the current descent direction and past multi-step iterative information as a new search direction and to obtain a stepsize by virtue of two types of strategies. It is proved that the developed direction with suitable parameters always satisfies the sufficient descent condition at each iteration. Based on mild assumptions, we obtain the global convergence and the rates of convergence for our method. Computational experiments are given to demonstrate the effectiveness of the proposed method.
\end{abstract}


\end{frontmatter}


\section{Introduction}
Many problems in space exploration, engineering design, finance, environment analysis, management and machine learning have several objectives to be minimized simultaneously \cite{T_a2004,M_s2004,Z_m2015,F_o2001,H_a2020,J_m2006}. This type of problems can be expressed as \emph{multiobjective optimization problems} (MOPs). The simultaneous minimization of multiple objectives differs from scalar optimization in that there is no unique solution to MOPs. An useful notion of optimality in multiobjective optimization is \emph{Pareto optimality}.

The development of solution schemes for solving MOPs has attracted a great amount of attention. Classical methods for solving MOPs include heuristic methods \cite{C_e2007,G_m2008} and scalarization methods \cite{M_n2000, E_a2008}. For the heuristic methods, there is no convergence theories and the empirical convergence is usually slow. The scalarization methods convert a given MOP into a parameterized scalar one. In general, the converted problem and the primal MOP enjoy same optimal solutions under certain conditions. Nevertheless, Fliege el al. \cite{F_n2009} pointed out that the parameters in scalarization method are not known in advance and the selection of parameters may result in unbounded scalar problems even if the original MOP has solutions. In order to cope with these limitations, the iterative methods for solving MOPs have been proposed by many researchers \cite{F_n2009,F_s2000,W_e2019,C_c2009,A_c2021,M_b2016,da_a2013,L_n2018,G_o2020,G_a2021,M_q2018,A_a2015,M_n2019}, which are deemed as extensions of scalar optimization methods. It should be pointed out that the iterative methods for solving MOPs have quite satisfactory convergence properties and are easy to implement. Extending the iterative methods in scalar optimization to multiobjective setting is currently a promising area of research.

The steepest descent method for solving MOPs was proposed by Fliege and Svaiter \cite{F_s2000}, which produces a sequence of iterates by the following update rule
\begin{equation}\label{sd_iter}
	x^{k+1}=x^{k}+\alpha_{k}d^{k},\quad k=0,1,2,\ldots,
\end{equation}
where $d^{k}=v(x^{k})\in\mathbb{R}^{n}$ is obtained by solving a auxiliary and non-parametric scalar subproblem at each iteration, which is called the descent direction, and $\alpha_{k}>0$ is the stepsize found by using a stepsize strategy. Fliege et al. \cite{F_c2019} pointed out that the rates of steepest descent method for smooth MOPs are consistent with that of scalar optimization, but numerical experiments presented in \cite{G_o2020,G_a2021} illustrate the unsatisfactory performance of the multiobjective steepest descent method in view of the computational efficiency on some test problems. The Newton method \cite{F_n2009} for MOPs employs the Hessian information of every objectives at each iteration and has faster convergence speed. However, the computational costs of the Hessian matrices are quite high-priced  for high dimensional problems, as reported in \cite{M_q2018}. We observe that the multiobjective steepest descent method \cite{F_s2000} only needs the current descent direction information explicitly to obtain the next iterative points, which undoubtedly leads to the waste of historical iterative information to a certain extent. As we all known, in scalar optimization, the nonlinear conjugate gradient methods utilizing past one-step information can accelerate the classical gradient method and avoid the computation of Hessian matrices \cite{S_op2006}. Borrowing the idea of nonlinear conjugate gradient methods in scalar optimization, Lucambio P\'{e}rez et al. \cite{L_n2018} proposed the multiobjective versions of such methods, which use the past one-step descent direction information to produce the next iterative point. It has the iterative form \eqref{sd_iter} with the search direction
\begin{equation}\label{cg_iter}
	d^{k}=\left\{
	\begin{aligned}
		&v(x^{k}), &\quad{\rm if}\quad k= 0,\\
		&v(x^{k})+\beta_{k}d^{k-1},& \quad{\rm if}\quad k\geq 1,
	\end{aligned}\right.
\end{equation}
where $\beta_{k}$ is a scalar algorithmic parameter. In \cite{L_n2018}, $\beta_{k}$ was considered as the extended versions of five classical parameters in scalar optimization. Moreover, the multiobjective extensions of Hager--Zhang and the Liu--Storey nonlinear conjugate gradient methods were respectively proposed in \cite{G_o2020} and \cite{G_a2021}. The global convergence of these methods in \cite{L_n2018,G_o2020,G_a2021} were analyzed under mild assumptions, but the convergence rates were not obtained. Numerical experiments provided in \cite{G_o2020,G_a2021} illustrate the multiobjective nonlinear conjugate gradient methods is superior than the multiobjective steepest descent method. This means that the use of past one-step iterative information improves the performance of algorithm in terms of computational efficiency to some great extent. Now, the question is whether we can design a new form of $d^{k}$ by considering the historical multi-step iterative information and further improve the algorithmic performance?

In scalar optimization, it is worth noting that there have been a number of meaningful researches which use the past multi-step iterative information to design algorithms. The momentum method introduced by Polyak \cite{P_s1964} can accelerate gradient method by combining the historical gradients information in the update rule at every iteration. It is widely utilized to train the parameters of neural network in machine learning \cite{S_o2013,L_a2020}. However, its convergence can not be guaranteed in general. Cragg and Levy \cite{C_s1969} introduced a method called supermemory gradient method to seek the minimum of a unconstrained optimization problem. Compared with the classical gradient method, their method memorizes the previous $k$-step iterations and has the advantage of high speed. Wolfe and Viazminsky \cite{W_s1976} studied a supermemory descent method that including the supermemory gradient method of Cragg and Levy \cite{C_s1969} as a special case. Nevertheless, the global convergence properties were not established in \cite{C_s1969,W_s1976}. Shi and Shen \cite{S_a2004,S_a2005} proposed a type of gradient-based algorithm whose basic idea is also to employ historical multi-step iterative direction information. Based on some suitable assumptions, the global convergence and the rate for convergence were obtained. Numerical experiments in \cite{S_a2004,S_a2005} reveal a fact that more information used in the current iterate may improve the algorithmic performance. Narushima and Yabe \cite{N_g2006} introduced a new memory gradient method that also uses historical direction information and then derived the global convergence of the method under appropriate conditions. Other methods that use historical iterative information at the current step to improve the algorithmic performance have been reported in 
\cite{G_a2011,G_a2014,O_a2017,Z_a2012}. In summary, it would be a good choice to design new algorithms based on historical iterative information in scalar optimization. So far as we know, there is no study on utilizing this strategy to design algorithm in multiobjective optimization.

Motivated by the works \cite{F_s2000, L_n2018,N_g2006}, the main goal of this paper is to introduce and analyze a new descent method called multiobjective memory gradient method for solving MOPs. In this approach, the direction $d^{k}$ is developed by the current descent direction and historical multi-step iterative information. We point out that the combined direction with appropriate parameters always meets the sufficient descent condition proposed in \cite{L_n2018}. The new method is considered with two classes of stepsize strategies for obtaining a stepsize along the search direction $d^{k}$. Based on several suitable assumptions, we derive the global convergence properties of the method. It is proved that the algorithm enjoys a convergence rate with the order of $1/\sqrt{k}$ to non-convex MOPs. We also give a new and reasonable assumption, and then establish the linear convergence rate under such assumption. As for the numerical experiments, we compare our method with the steepest descent method \cite{F_s2000} and the nonlinear conjugate gradient method \cite{L_n2018}  on a set of test instances. The numerical results illustrate our method's effectiveness, as will be presented in Section 6.

The remainder of this paper is as follows. Section 2 presents some notions, notation and preliminary results. In Section 3, we give the general scheme of memory gradient method for solving MOPs and provide the descent properties of the combined direction. In Section 4, we conduct the convergence analysis of the proposed method with two different stepsize strategies. In Section 5, we prove the convergence rates. Numerical experiments are provided in Section 6. Finally, in Section 7, we make some conclusions about our works.

\section{Preliminaries}
Throughtout this paper, for $m\in\mathbb{Z}_{+}$, where $\mathbb{Z}_{+}$ denotes the set of positive integers, we take $\langle m\rangle=\{1,2,\ldots,m\}$. Denote $e=(1,1,\ldots,1)^{\top}\in\mathbb{R}^{m}$. Let $0_{n}$ denote the zero vector of $\mathbb{R}^{n}$. $\langle \cdot,\cdot\rangle$ stands for the inner product in $\mathbb{R}^{n}$ and $\|\cdot\|$ denotes the norm, that is $\|x\|=\sqrt{\langle x,x\rangle}$ for $x\in\mathbb{R}^{n}$. The norm of a real matrix $A=(A_{i,j})_{m\times n}\in\mathbb{R}^{m\times n}$ is calculated as
\begin{equation}\label{matrix_norm}
	\|A\|=\max_{x\neq0}\frac{\|Ax\|_{\infty}}{\|x\|}=\max_{i\in\langle m\rangle}\|A_{i,\cdot}\|=\max_{i\in\langle m\rangle}
	\left(\sum_{j=1}^{n}A_{i,j}^{2}\right)^{1/2}.
\end{equation}

Let $\chi^{+}$ ($\chi\in\mathbb{R}$) be denoted as
\begin{equation*}
	\chi^{+}=\left\{
	\begin{aligned}
		&0, &\quad{\rm if}\quad \chi= 0,\\
		&\frac{1}{\chi},& \quad {\rm otherwise}.
	\end{aligned}\right.
\end{equation*}
Clearly, $\chi\chi^{+}\leq1$ and $\chi\chi^{+}=1$ only when $\chi\neq0$.

Denote $\mathbb{R}_{+}^{m}=\{\mu\in\mathbb{R}^{m}:\mu_{i}\geq0,i\in\langle m\rangle\}$ and $\mathbb{R}_{++}^{m}=\{\mu\in\mathbb{R}^{m}:\mu_{i}>0,i\in\langle m\rangle\}$. We define the partial order $\preceq$ induced by $\mathbb{R}_{+}^{m}$: for $\nu,\mu\in\mathbb{R}^{n}$, $\nu\preceq \mu$ if and only if (iff) $\mu-\nu\in\mathbb{R}_{+}^{m}$, which is equivalent to $\nu_{i}\leq \mu_{i}$ for each $i\in\langle m\rangle$. Likewise, we also introduce the stronger relation $\prec$ induced by $\mathbb{R}_{++}^{m}$: $\nu\prec \mu$ iff $\mu-\nu\in\mathbb{R}_{++}^{m}$, which is equivalent to $\nu_{i}< \mu_{i}$ for each $i\in\langle m\rangle$.  Sometimes we may use $\mu\succeq \nu$ instead of $\nu\preceq \mu$. Let $-\mathbb{R}_{++}^{m}$ be the negative of $\mathbb{R}_{++}^{m}$, that is, $-\mathbb{R}_{++}^{m}=\{-\mu:\mu\in\mathbb{R}_{++}^{m}\}$.

In this paper, we are concerned with the following MOP:
\begin{equation}\label{mop}
	\min_{x\in\mathbb{R}^{n}}\quad F(x)=(F_{1}(x),F_{2}(x),...,F_{m}(x))^{\top},\\
\end{equation}
where $F_{i}:\mathbb{R}^{n}\rightarrow\mathbb{R}$, $i\in\langle m\rangle$, are continuously differentiable and the superscript ``${\top}$'' denotes the transpose. A point $\bar{x}\in\mathbb{R}^{n}$ is said to be \emph{Pareto optimal} of problem \eqref{mop} if there exists no $x\in\mathbb{R}^{n}$ such that $F(x)\preceq F(\bar{x})$ and $F(x)\neq F(\bar{x})$ (see \cite{M_n2000}).

Given $x=(x_{1},x_{2},\ldots,x_{n})\in\mathbb{R}^{n}$, the Jacobian of $F$ at $x$ is defined by 
$$JF(x)=[\nabla F_{1}(x)\:\:\:\nabla F_{2}(x)\:\:\:\ldots\:\:\:\nabla F_{m}(x)]^{\top}.$$
The image of $JF(x)$ is denoted as
$${\rm Im}(JF(x))=\{JF(x)d=(\langle\nabla F_{1}(x),d\rangle,\langle\nabla F_{2}(x),d\rangle\ldots,\langle\nabla F_{m}(x),d\rangle)^{\top}:d\in\mathbb{R}^{n}\}.$$

A first-order necessary condition introduced in \cite{F_s2000} for Pareto optimality of a point $\bar{x}\in\mathbb{R}^{n}$ is
\begin{equation*}\label{fir_ord_opt}
	{\rm Im}(JF(\bar{x}))\cap(-\mathbb{R}^{m}_{++})=\emptyset.
\end{equation*}
A point $\bar{x}\in\mathbb{R}^{n}$ satisfying the above relation is said to be \emph{Pareto critical} (see \cite{F_s2000}).
Equivalently, for any $d\in\mathbb{R}^{n}$, there is $i^{*}\in\langle m\rangle$ such that $(JF(\bar{x})d)_{i^{*}}=\langle\nabla F_{i^{*}}(\bar{x}),d\rangle\geq0$,
which implies $\max_{i\in\langle m\rangle}\langle\nabla F_{i}(\bar{x}),d\rangle\geq0$ for any $d\in\mathbb{R}^{n}$. 
Clearly, if $x\in\mathbb{R}^{n}$ is not a Pareto critical point, then there is a vector $d\in\mathbb{R}^{n}$ satisfying $JF(x)d\in-\mathbb{R}^{m}_{++}$. We call the vector $d$ a \emph{descent direction} for $F$ at $x$.

Now, we define $\psi:\mathbb{R}^{n}\times\mathbb{R}^{n}\rightarrow \mathbb{R}$ as follows:
\begin{equation}\label{h}
	\psi(x,d)=\max_{i\in\langle m\rangle}\langle\nabla F_{i}(x),d\rangle.
\end{equation}
From the previous discussion, we known that $\psi$ can express Pareto critical and descent direction, i.e.,
\begin{itemize}\setlength{\itemsep}{-0.01in}
	\item $d\in\mathbb{R}^{n}$ is a descent direction for $F$ at $x\in\mathbb{R}^{n}$ iff $\psi(x,d)<0$,
	\item $x\in\mathbb{R}^{n}$ is Pareto critical iff $\psi(x,d)\geq0$ for any $d\in\mathbb{R}^{n}$. 
\end{itemize}

The following proposition illustrates several useful results related to $\psi$.

\begin{proposition}{\rm\cite{F_a2014}}\label{h_property}
	For all $x,y\in\mathbb{R}^{n}$, $\varrho> 0$ and $b_{1},b_{2}\in\mathbb{R}^{n}$, we obtain
	\begin{enumerate}[{\rm (i)}]\setlength{\itemsep}{-0.01in}
		\item $\psi(x,\varrho b_{1})=\varrho \psi(x,b_{1})$;
		\item $\psi(x,b_{1}+b_{2})\leq \psi(x,b_{1})+\psi(x,b_{2})$;
		\item $|\psi(x,b_{1})-\psi(y,b_{2})|\leq\|JF(x)b_{1}-JF(y)b_{2}\|$.
	\end{enumerate}
\end{proposition}

Let us now consider the following scalar optimization problem:
\begin{equation}\label{sub_pro}
	\min_{d\in\mathbb{R}^{n}}\:\psi(x,d)+\frac{1}{2}\|d\|^{2}.
\end{equation}
Obviously, the objectives in \eqref{sub_pro} is proper, closed and strongly convex. Therefore, problem \eqref{sub_pro} admits a unique optimal solution. Denote the optimal solution of (\ref{sub_pro}) by $v(x)$, i.e.,
\begin{equation}\label{opt_sol}
	v(x)=\mathop{\rm argmin}_{d\in\mathbb{R}^{n}}\: \psi(x,d)+\frac{1}{2}\|d\|^{2},
\end{equation}
and let the optimal value of (\ref{sub_pro}) be defined as $\theta(x)$, i.e., 
\begin{equation}\label{opt_val}
	\theta(x)=\psi(x,v(x))+\frac{1}{2}\|v(x)\|^{2}.
\end{equation}
Observe that in scalar optimization (i.e., $m=1$), one has $\psi(x,d)= \langle\nabla F_{1}(x),d\rangle$, $v(x)=-\nabla F_{1}(x)$ and $\theta(x)=-\|\nabla F_{1}(x)\|^{2}/2$.

To obtain $v(x)$, one can consider the corresponding dual problem of \eqref{sub_pro} (see \cite{F_s2000}):

\begin{equation}\label{dual}
	\begin{aligned}
		\lambda(x)\in\mathop{\rm argmin}_{\lambda\in\mathbb{R}^{m}}&\quad \left\|\sum_{i=1}^{m}\lambda_{i}\nabla F_{i}(x)\right\|^{2}\\
		{\rm s.t.}&\quad \lambda\in\varLambda^{m},
	\end{aligned}
\end{equation}
where $\varLambda^{m}=\{\lambda\in\mathbb{R}^{m}:\sum_{i=1}^{m}\lambda_{i}=1,\lambda_{i}\geq0,\forall i\in\langle m\rangle\}$ stands for the simplex set. Then, $v(x)$ can also be represented as
\begin{equation}\label{dual_sol}
	v(x)=-\sum_{i=1}^{m}(\lambda(x))_{i}\nabla F_{i}(x).
\end{equation}

Let us now give a characterization of Pareto critical points of problem \eqref{mop}, which will be used in our subsequent analysis.

\begin{proposition}{\rm\cite{F_s2000}}\label{pa_sta_equ}
	Let $v(\cdot)$ and $\theta(\cdot)$ be as in \eqref{opt_sol} and \eqref{opt_val}, respectively. The following statements hold:
	\begin{enumerate}[{\rm (i)}]\setlength{\itemsep}{-0.01in}
		\item If $x$ is a Pareto critical point of problem \eqref{mop}, then $v(x)= 0$ and $\theta(x)=0$;
		\item If $x$ is not a Pareto critical point of problem \eqref{mop}, then $v(x)\neq 0$, $\theta(x)<0$ and $\psi(x,v(x))<-\|v(x)\|^{2}/2<0$;
		\item $v(\cdot)$ is continuous.
	\end{enumerate}
\end{proposition}

\section{Multiobjective memory gradient method}

In what follows, we introduce the multiobjective memory gradient (MMG) algorithm which uses the previous iterative information for solving the problem \eqref{mop}.

\begin{newlist}{Step1: }\setlength{\itemsep}{-0.01in}
	\item[MMG algorithm.]
	\item[Step 0] Choose $x^{0}\in\mathbb{R}^{n}$, $\gamma_{0}>0$ and $N\in\mathbb{Z}_{+}$. Compute $v(x^{0})$ and initialize $k\leftarrow0$.
	\item[Step 1] If $v(x^{k})=0$, then STOP.
	\item[Step 2] Define
	\begin{equation}\label{dk}
		d^{k}=\left\{
		\begin{aligned}
			&\gamma_{k}v(x^{k}), &\quad{\rm if}\quad k= 0,\\
			&\gamma_{k}v(x^{k})+\sum_{j=1}^{N_{k}}\beta_{kj}d^{k-j},& \quad{\rm if}\quad k\geq 1,
		\end{aligned}\right.
	\end{equation}
	where $\gamma_{k}>0$ and $\beta_{kj}\in\mathbb{R}$ ($j\in\langle N_{k}\rangle$, $N_{k}=\min\{k, N\}$) are algorithmic parameters.
	\item[Step 3] Find a stepsize $\alpha_{k}>0$ by a stepsize strategy and set $x^{k+1}=x^{k}+\alpha_{k}d^{k}$.
	\item[Step 4] Compute $v(x^{k+1})$, do $k\leftarrow k+1$, and return to Step 1.
\end{newlist}

From \eqref{dk}, it follows that the previous direction information are merged into the current descent direction by virtue of the parameters $\beta_{kj}$ ($j\in\langle N_{k}\rangle$).  The selection for updating the parameters $\gamma_{k}$ and $\beta_{kj}$ ($j\in\langle N_{k}\rangle$) at Step 2 and the way for obtaining the stepsize $\alpha_{k}$ at Step 3 remain deliberately open. 
We emphasize that if $N=1$ and $\gamma_{k}=1$ for each $k$, then the MMG algorithm reduces to the framework of the multiobjective nonlinear conjugate gradient algorithm; see \cite[pp.2699--2700]{L_n2018}. If $N=1$, $\gamma_{k}=1$ for each $k$ and $\beta_{kj}=0$ for all $k\geq1$, then \eqref{dk} can be viewed as the iterative form of the multiobjective steepest descent method \cite{F_s2000}. 

To ensure the convergence of the proposed method, we present a reasonable choice of parameters in the sequel and consider two appropriate stepsize strategies in the next section. It is obvious from the algorithmic framework that the MMG algorithm can successfully terminate when a Pareto critical point of problem \eqref{mop} is obtained. Thus, we assume that $v(x^{k})\neq0$ for any $k\geq0$ in the subsequent analysis. This means that the MMG algorithm generates infinite sequences $\{x^{k}\}$ and $\{d^{k}\}$.

Define $\beta_{kj}$ ($j\in\langle N_{k}\rangle$) as follows
\begin{equation}\label{betakj}
	\beta_{kj}=-\frac{1}{N_{k}}\psi(x^{k},v(x^{k}))\phi_{kj}^{+},
\end{equation}
where $\phi_{kj}$ ($j\in\langle N_{k}\rangle$) are parameters satisfying the following relation:
\begin{equation}\label{psikj}
	\phi_{kj}>\max\left\{\frac{\psi(x^{k},d^{k-j})}{\gamma_{k}},0\right\}.
\end{equation}
Note that $\beta_{kj}>0$ since $\psi(x^{k},v(x^{k}))<0$ for any $k$.

The next property displays that $d^{k}$ is a descent direction when the related parameters satisfy \eqref{betakj} and \eqref{psikj}.

\begin{lemma}\label{sufficient_des}
	Let the direction $d^{k}$ be given in (\ref{dk}). Assume that $\beta_{kj}$ and $\phi_{kj}$ satisfy \eqref{betakj}
	and \eqref{psikj} for $k\geq1$ and $j\in\langle N_{k}\rangle$, respectively. Then, $d^{k}$ is a descent direction for all $k$.
\end{lemma}

\begin{proof}
	Consider $k=0$. By Proposition \ref{h_property}(i) and Proposition \ref{pa_sta_equ}(ii), we have $$\psi(x^{0},d^{0})=\gamma_{0}\psi(x^{0},v(x^{0}))<-\frac{\gamma_{0}}{2}\|v(x^{0})\|^{2}<0.$$
	For $k\geq1$, according to \eqref{dk} and Proposition \ref{h_property}(i)--(ii), one has
	\begin{equation}\label{sufficient_des00}
		\psi(x^{k},d^{k})=\psi\left(x^{k}, \gamma_{k}v(x^{k})+\sum_{j=1}^{N_{k}}\beta_{kj}d^{k-j}\right)
		\leq \gamma_{k}\psi(x^{k},v(x^{k}))+\psi\left(x^{k},\sum_{j=1}^{N_{k}}\beta_{kj}d^{k-j}\right).
	\end{equation}
	Applying Proposition \ref{h_property}(i)--(ii) repeatedly on the right part of \eqref{sufficient_des00}, it follows that
	\begin{equation}\label{sufficient_des0}
		\psi(x^{k},d^{k})\leq \gamma_{k}\psi(x^{k},v(x^{k}))+\sum_{j=1}^{N_{k}}\beta_{kj}\psi(x^{k},d^{k-j}).
	\end{equation}
	Consider the last term of \eqref{sufficient_des0}. By the definition of $\beta_{kj}$ ($j\in\langle N_{k}\rangle$), we have
	\begin{equation}\label{sufficient_des1}
		\begin{aligned}
			\beta_{kj}\psi(x^{k},d^{k-j})&\leq\beta_{kj}\max\{\psi(x^{k},d^{k-j}),0\}\\
			&=-\frac{1}{N_{k}}\psi(x^{k},v(x^{k}))\phi_{kj}^{+}\max\{\psi(x^{k},d^{k-j}),0\}\\
			&<-\frac{1}{N_{k}}\gamma_{k}\psi(x^{k},v(x^{k}))\phi_{kj}^{+}\phi_{kj}\\
			&\leq-\frac{1}{N_{k}}\gamma_{k}\psi(x^{k},v(x^{k})).
		\end{aligned}
	\end{equation}
	By \eqref{sufficient_des0} and \eqref{sufficient_des1}, we obtain $\psi(x^{k},d^{k})<0$.
\end{proof}

For the subsequent convergence analysis to MMG algorithm, we will need the more stringent condition
\begin{equation}\label{suf_des_con}
	\psi(x^{k},d^{k})\leq \sigma \psi(x^{k},v(x^{k})),
\end{equation}
for some $\sigma>0$ and any $k\geq0$. Notice that in scalar optimization, the condition \eqref{suf_des_con} becomes
\begin{equation*}
	\langle \nabla F_{1}(x^{k}),d^{k}\rangle\leq -\sigma \|\nabla F_{1}(x^{k})\|^{2},
\end{equation*}
which is the well-known \emph{sufficient descent condition}. Likewise, in multiobjective optimization, we say that a direction $d^{k}\in\mathbb{R}^{n}$ meets the sufficient descent condition at $x^{k}$ iff \eqref{suf_des_con} holds. It is worth mentioning that the general concept for sufficient descent condition was first proposed by Lucambio P${\rm\acute{e}}$rez and Prudente \cite{L_n2018} in vector optimization. The general notion was also applied in \cite{L_n2018, G_o2020, G_a2021} to discuss the convergence of nonlinear conjugate gradient methods for vector optimization. 

In the next lemma, we present the sufficient descent property on $d^{k}$ under stronger conditions on $\gamma_{k}$ and $\phi_{kj}$ ($j\in\langle N_{k}\rangle$).

\begin{lemma}\label{sufficient_descent}
	Let the direction $d^{k}$ be given in (\ref{dk}). Suppose that there is a positive constant $\gamma^{*}$ such that $\gamma_{k}\geq\gamma^{*}$, $\beta_{kj}$ satisfies \eqref{psikj} and $\phi_{kj}$ has the following property:
	\begin{equation}\label{sufficient_descent0001}
		\phi_{kj}>\frac{\psi(x^{k},d^{k-j})+\|JF(x^{k})\|\|d^{k-j}\|}{\gamma_{k}}.
	\end{equation}
	Then, $d^{k}$ satisfies the sufficient descent condition \eqref{suf_des_con} with $\sigma=\gamma^{*}/2>0$ for any $k$.
\end{lemma}

\begin{proof}
	Consider $k=0$. By Proposition \ref{h_property}(i) and $\psi(x^{0},v(x^{0}))<0$, one has
	\begin{equation*}
		\psi(x^{0},d^{0})=\gamma_{0}\psi(x^{0},v(x^{0}))\leq \gamma^{*}\psi(x^{0},v(x^{0}))\leq\frac{\gamma^{*}}{2}\psi(x^{0},v(x^{0})).
	\end{equation*}
	For $k\geq1$, from \eqref{sufficient_des0}, we have
	\begin{equation}\label{sufficient_descent0}
		\psi(x^{k},d^{k})\leq \gamma_{k}\psi(x^{k},v(x^{k}))+\sum_{j=1}^{N_{k}}\beta_{kj}\psi(x^{k},d^{k-j}).
	\end{equation}
	It follows from $\psi(x^{k},0_{n})=0$ and Proposition \ref{h_property}(iii) that
	\begin{equation}\label{sufficient_descent000}
		|\psi(x^{k},d^{k-j})|=|\psi(x^{k},d^{k-j})-\psi(x^{k},0_{n})|\leq\|JF(x^{k})d^{k-j}-0_{m}\|\leq\|JF(x^{k})\|\|d^{k-j}\|.
	\end{equation}
	Therefore, we conclude that $\phi_{kj}>0$ and $\beta_{kj}>0$. Now we define the index set $\mathcal{J}=\{j\in\langle N_{k}\rangle:\psi(x^{k},d^{k-j})>0\}$. Clearly, $|\mathcal{J}|\in\{0,1,\ldots,N_{k}\}$. We next consider two cases.
	
	\emph{Case 1}. If $|\mathcal{J}|=0$, then \eqref{sufficient_descent0} implies
	\begin{equation}\label{sufficient_descent01}
		\psi(x^{k},d^{k})\leq \gamma_{k}\psi(x^{k},v(x^{k}))\leq\gamma^{*}\psi(x^{k},v(x^{k}))\leq\frac{\gamma^{*}}{2}\psi(x^{k},v(x^{k})),
	\end{equation}
	because $\psi(x^{k},v(x^{k}))<0$. 
	
	\emph{Case 2}. If $|\mathcal{J}|\neq0$, then for the last term of \eqref{sufficient_descent0}, we have
	\begin{equation}\label{sufficient_descent1}
		\begin{aligned}
			\sum_{j=1}^{N_{k}}\beta_{kj}\psi(x^{k},d^{k-j})
			&\leq\sum_{j\in\mathcal{J}}\beta_{kj}\psi(x^{k},d^{k-j})\\
			&=\sum_{j\in\mathcal{J}}-\frac{1}{N_{k}}\psi(x^{k},v(x^{k}))\phi_{kj}^{+}\psi(x^{k},d^{k-j})\\
			&=\sum_{j\in\mathcal{J}}\frac{1}{N_{k}}\frac{-\psi(x^{k},v(x^{k}))\psi(x^{k},d^{k-j})}{\phi_{kj}}\\
			&\leq\sum_{j\in\mathcal{J}}\frac{1}{N_{k}}\frac{-\gamma_{k}\psi(x^{k},v(x^{k}))\psi(x^{k},d^{k-j})}{\psi(x^{k},d^{k-j})+\|JF(x^{k})\|\|d^{k-j}\|}\\
			&\leq-\frac{1}{N_{k}}\gamma_{k}\psi(x^{k},v(x^{k}))\sum_{j\in\mathcal{J}} \frac{\psi(x^{k},d^{k-j})}{2\psi(x^{k},d^{k-j})}\\
			&\leq-\frac{|\mathcal{J}| }{2N_{k}}\gamma_{k}\psi(x^{k},v(x^{k})),
		\end{aligned}
	\end{equation}
	where the first equality follows from the definition of $\beta_{kj}$ and the penultimate inequality follows from \eqref{sufficient_descent000}. By \eqref{sufficient_descent0} and \eqref{sufficient_descent1}, we obtain
	$$\psi(x^{k},d^{k})\leq \gamma_{k}\left(1-\frac{|\mathcal{J}|}{2N_{k}}\right)\psi(x^{k},v(x^{k}))\leq \frac{\gamma_{k}}{2}\psi(x^{k},v(x^{k}))\leq \frac{\gamma^{*}}{2}\psi(x^{k},v(x^{k})).$$
	Hence, let $\sigma=\gamma^{*}/2$ and the proof is complete.
\end{proof}

\section{Convergence analysis}

This section is devoted to the global convergence analysis of MMG algorithm with two different stepsize strategies. Let the following assumptions are satisfied.
\begin{description}
	\item[(A1)] $F$ is bounded below on the set $\mathcal{L}=\{x\in\mathbb{R}^{n}:F(x)\preceq F(x^{0})\}$, where $x^{0}\in\mathbb{R}^{n}$ is a given starting point.
	\item[(A2)] The Jacobian $JF$ is Lipschitz continuous with $L>0$ on an open set $\mathcal{B}$ containing $\mathcal{L}$, i.e., $\|JF(x)- JF(y)\|\leq L\|x-y\|$ for all $x,y\in\mathcal{B}$.
\end{description}

We are now ready to describe the stepsize strategies. 

\noindent\textbf{Stepsize-I strategy.} Let $\rho\in(0,1)$ and $\delta\in(0,1)$. Set $\tau_{k}=-\psi(x^{k},d^{k})/\|d^{k}\|^{2}$ and choose $\alpha_{k}=\max\{\tau_{k},\delta\tau_{k},\delta^{2}\tau_{k},\ldots\}$ satisfying
\begin{equation}\label{armijio}
	F(x^{k}+\alpha_{k} d^{k})\preceq F(x^{k})+\rho\alpha_{k} \psi(x^{k},d^{k})e.
\end{equation}

\noindent\textbf{Stepsize-II strategy.} Assume that $F$ satisfies (A2). Define the stepsize as 
\begin{equation}\label{adaptive}
	\alpha_{k}=\frac{-\psi(x^{k},d^{k})}{2L\|d^{k}\|^{2}}
\end{equation}
Obviously, the stepsize-I strategy is an Armijio-type line search and $\alpha_{k}$ is obtained by a simple backtracking procedure. The stepsize-II strategy depends on the Lipschitz constant $L$. If $L$ is not easily evaluated, then \eqref{adaptive} will not be calculated. This means that the stepsize rule \eqref{adaptive} is merely theoretical. Nevertheless, we notice that stepsize involving Lipschitz constant has been considered in the multiobjective setting (see \cite{G_a2021, A_c2021}).

For convenience, the MMG algorithm equipped with the stepsize-I strategy and the parameters satisfying the conditions of Lemma \ref{sufficient_descent} is identified as MMG-I. We also refer by MMG-II to the MMG algorithm equipped with the stepsize-II strategy and the parameters satisfying the conditions of Lemma \ref{sufficient_descent}.

The following two lemmas display the decrease property for the function value of iterate points generated by MMG-I and MMG-II, respectively.
\begin{lemma}\label{armijio_dk_bounds}
	Assume that the sequence $\{(x^{k},d^{k})\}$ is produced by MMG-I and that (A2) holds. Then, there is a positive constant $\omega$ such that
	\begin{equation}\label{armijio_dk_bounds0}
		F(x^{k})-F(x^{k+1})\succeq \omega\frac{\psi^{2}(x^{k},d^{k})}{\|d^{k}\|^{2}}e,\quad \forall k.
	\end{equation}
\end{lemma}

\begin{proof}
	For the stepsize-I strategy, we have the following two cases.
	
	\emph{Case 1.} Let $\mathcal{K}_{1}=\{k:\alpha_{k}=\tau_{k}\}$. By \eqref{armijio}, for every $i\in\langle m\rangle$ and all $k\in\mathcal{K}_{1}$, we have
	\begin{equation}\label{armijio_dk_bounds2}
		F_{i}(x^{k})-F_{i}(x^{k+1})\geq-\rho\tau_{k} \psi(x^{k},d^{k})=\rho \frac{\psi^{2}(x^{k},d^{k})}{\|d^{k}\|^{2}}.
	\end{equation}
	
	\emph{Case 2.} Let $\mathcal{K}_{2}=\{k:\alpha_{k}<\tau_{k}\}$. Obviously, $\alpha_{k}/\delta\leq\tau_{k}$ for $k\in\mathcal{K}_{2}$. Let $t=\alpha_{k}/\delta$. By the way $\alpha_{k}$ is chosen in the stepsize-I strategy, it follows that $t$ fails to
	satisfy \eqref{armijio}, i.e., for any $k\in\mathcal{K}_{2}$,
	\begin{equation*}
		F(x^{k}+t d^{k})\npreceq F(x^{k})+\rho t \psi(x^{k},d^{k})e,
	\end{equation*}
	which means that
	\begin{equation}\label{armijio_dk_bounds3}
		F_{i_{k}}(x^{k}+t d^{k})- F_{i_{k}}(x^{k})>\rho t \psi(x^{k},d^{k})
	\end{equation}
	for at least one index $i_{k}\in\langle m\rangle$. Applying the mean value theorem on the left part of \eqref{armijio_dk_bounds3}, there is $\nu_{k}\in[0,1]$ such that
		$\langle\nabla F_{i_{k}}(x^{k}+t\nu_{k}d^{k}),d^{k}\rangle>\rho \psi(x^{k},d^{k})$ for any $k\in\mathcal{K}_{2}$.
	By the definition of $\psi(\cdot,\cdot)$, we get $\psi(x^{k},d^{k})\geq \langle\nabla F_{i_{k}}(x^{k}),d^{k}\rangle$. Combining this with the above inequality, for all $k\in\mathcal{K}_{2}$, we have
	\begin{equation}\label{armijio_dk_bounds5}
		\begin{aligned}
			(\rho-1)\psi(x^{k},d^{k})&<\langle\nabla F_{i_{k}}(x^{k}+t\nu_{k}d^{k}),d^{k}\rangle-\psi(x^{k},d^{k})\\
			&\leq\langle\nabla F_{i_{k}}(x^{k}+t\nu_{k}d^{k}),d^{k}\rangle-\langle\nabla F_{i_{k}}(x^{k}),d^{k}\rangle.
		\end{aligned}
	\end{equation}
	By Cauchy-Schwarz inequality, (A2) and the fact that $\nu_{k}\in[0,1]$, for every $k\in\mathcal{K}_{2}$, we get
	\begin{equation}\label{armijio_dk_bounds6}
		\begin{aligned}
			\langle\nabla F_{i_{k}}(x^{k}+t\nu_{k}d^{k}),d^{k}\rangle-\langle\nabla F_{i_{k}}(x^{k}),d^{k}\rangle&\leq\|\nabla F_{i_{k}}(x^{k}+t\nu_{k}d^{k})-\nabla F_{i_{k}}(x^{k})\|\|d^{k}\|\\
			&\leq L t\nu_{k}\|d^{k}\|^{2}\\
			&\leq L t\|d^{k}\|^{2}.
		\end{aligned}
	\end{equation}
	Therefore, by \eqref{armijio_dk_bounds5} and \eqref{armijio_dk_bounds6}, for each $k\in\mathcal{K}_{2}$, we immediately have 
	\begin{equation}\label{armijio_dk_bounds7}
		\alpha_{k}=t\delta\geq \frac{\delta(\rho-1)}{L} \frac{\psi(x^{k},d^{k})}{\|d^{k}\|^{2}}.
	\end{equation} 
	By \eqref{armijio} and \eqref{armijio_dk_bounds7}, and noting that $\psi(x^{k},d^{k})<0$, one has
	\begin{equation}\label{armijio_dk_bounds8}
		F_{i}(x^{k})-F_{i}(x^{k+1})\geq-\rho\alpha_{k}\psi(x^{k},d^{k})\geq\frac{\rho\delta(1-\rho)}{L} \frac{\psi^{2}(x^{k},d^{k})}{\|d^{k}\|^{2}}
	\end{equation}
	for each $i\in\langle m\rangle$ and all $k\in\mathcal{K}_{2}$.
	
	Consequently, if we set 
	\begin{equation}\label{ww}
		\omega=\min\left\{\rho,\frac{\rho\delta (1-\rho)}{L}\right\},
	\end{equation}
	then the desired result \eqref{armijio_dk_bounds0} is satisfied. 
\end{proof}

\begin{lemma}\label{armijio_dk_bounds11}
	Assume that the sequence $\{(x^{k},d^{k})\}$ is produced by MMG-II. Then
	\begin{equation}\label{armijio_dk_bounds110}
		F(x^{k})-F(x^{k+1})\succeq \frac{1}{4L}\frac{\psi^{2}(x^{k},d^{k})}{\|d^{k}\|^{2}}e,\quad \forall k.
	\end{equation}
\end{lemma}

\begin{proof}
	From the mean value theorem, for every $i\in\langle m\rangle$ and any $k$, there is $t_{k}\in[0,1]$ such that
	\begin{equation}\label{armijio_dk_bounds1100}
		F_{i}(x^{k+1})-F_{i}(x^{k})=\langle\nabla F_{i}(x^{k}+t_{k}\alpha_{k}d^{k}),\alpha_{k}d^{k}\rangle.
	\end{equation}
	By Cauchy-Schwarz inequality and (A2), we get
	\begin{equation}\label{armijio_dk_bounds1101}
		\begin{aligned}
			\langle\nabla F_{i}(x^{k}+t_{k}\alpha_{k}d^{k}),d^{k}\rangle&=\langle\nabla F_{i}(x^{k}),d^{k}\rangle+\langle\nabla F_{i}(x^{k}+t_{k}\alpha_{k}d^{k})-\nabla F_{i}(x^{k}),d^{k}\rangle \\
			&\leq\langle\nabla F_{i}(x^{k}),d^{k}\rangle+\|\nabla F_{i}(x^{k}+t_{k}\alpha_{k}d^{k})-\nabla F_{i}(x^{k})\|\|d^{k}\|\\
			&\leq\langle\nabla F_{i}(x^{k}),d^{k}\rangle+Lt_{k}\alpha_{k}\|d^{k}\|^{2}\\
			&\leq\langle\nabla F_{i}(x^{k}),d^{k}\rangle+L\alpha_{k}\|d^{k}\|^{2}\\
			&=\langle\nabla F_{i}(x^{k}),d^{k}\rangle-\frac{1}{2}\psi(x^{k},d^{k})\\
			&\leq \psi(x^{k},d^{k})-\frac{1}{2}\psi(x^{k},d^{k})\\
			&=\frac{1}{2}\psi(x^{k},d^{k}),
		\end{aligned}
	\end{equation}
	where the last inequality holds because the definition of $\psi(\cdot,\cdot)$ as in \eqref{h}. Thus, by \eqref{armijio_dk_bounds1100}, \eqref{armijio_dk_bounds1101} and the definitions of $\alpha_{k}$ given in \eqref{adaptive}, we have
	\begin{equation}
		F_{i}(x^{k+1})-F_{i}(x^{k})\leq\frac{1}{2}\alpha_{k} \psi(x^{k},d^{k})=-\frac{1}{4L}\frac{\psi^{2}(x^{k},d^{k})}{\|d^{k}\|^{2}},\quad \forall i\in\langle m\rangle,
	\end{equation}
	which concludes the proof.
\end{proof}

According to \eqref{armijio_dk_bounds0} (or \eqref{armijio_dk_bounds110}) and (A1), we obtain that $\{F(x^{k})\}_{k\geq0}$ is monotone non-increasing and bounded below, hence convergent. We further have the following lemma.

\begin{lemma}\label{armijio_vk}
	Assume that the sequence $\{(x^{k},d^{k})\}$ is generated by MMG-I or MMF-II. If (A1) and (A2) are satisfied, then
	\begin{equation}\label{armijio_vk1}
		\sum_{k=0}^{\infty}\frac{\psi^{2}(x^{k},d^{k})}{\|d^{k}\|^{2}}<\infty.
	\end{equation}
\end{lemma}

In the following, we will prove the global convergence of MMG-I/MMG-II. To this aim, (A1) needs to be replaced by the following stronger assumption.

\begin{description}\setlength{\itemsep}{-0.05in}
	\item[(A3)] The set $\mathcal{L}$ is bounded.
\end{description}

\begin{remark}\normalfont\label{vkdk}
	Based on (A3), the sequence $\{\langle v(x^{k}), d^k\rangle\}$ is bounded. Indeed, it follows from Lemma \ref{armijio_dk_bounds} (or Lemma \ref{armijio_dk_bounds11}) and (A3) that $\{x^{k}\}$ is contained in the bounded set $\mathcal{L}$. By Proposition \ref{pa_sta_equ}(iii), we have $\{v(x^{k})\}$ is bounded, and thus $\{d^{k}\}$ is also bounded. That is, there exist constants $\xi_{1},\xi_{2}>0$ such that $\|v(x^{k})\|\leq\xi_{1}$ and $\|d^{k}\|\leq\xi_{2}$. Therefore, $\langle v(x^{k}), d^k\rangle\leq\|v(x^{k})\|\|d^{k}\|\leq\xi_{1}\xi_{2}$.
\end{remark}

\begin{theorem}\label{convergence_analysis1}
	Assume that the sequence $\{(x^{k},d^{k})\}$ is produced by MMG-I or MMG-II. If (A2) and (A3) hold, then $\lim\inf_{k\rightarrow\infty}\|v(x^{k})\|=0$. 
\end{theorem}

\begin{proof}
	Suppose by contradiction that there is $\varsigma>0$ such that
	$$\|v(x^{k})\|\geq\varsigma$$
	for any $k$. From $\psi(x^{k},0_{n})=0$, Proposition \ref{h_property}(iii) and Proposition \ref{pa_sta_equ}(ii), we obtain
	\begin{equation*}
		\dfrac{\|v(x^{k})\|^{2}}{2}\leq \psi(x^{k},0_{n}) -\psi(x^{k},v(x^{k}))\leq\|0_{m}-JF(x^{k})v(x^{k})\|\leq\|JF(x^{k})\|\|v(x^{k})\|,
	\end{equation*}
	which implies that
	\begin{equation}\label{convergence_analysis0000}
		\|v(x^{k})\|\leq2\|JF(x^{k})\|.
	\end{equation}
	According to the definition of $\beta_{kj}$, one has
	\begin{equation*}\label{convergence_analysis00000}
		\gamma_{k}\beta_{kj}\phi_{kj}=-\frac{1}{N_{k}}\gamma_{k}\psi(x^{k},v(x^{k}))\phi_{kj}^{+}\phi_{kj}=-\frac{1}{N_{k}}\gamma_{k}\psi(x^{k},v(x^{k})).
	\end{equation*}
	This, together with \eqref{sufficient_descent0} and \eqref{convergence_analysis0000}, yields
	\begin{equation}\label{convergence_analysis00}
		\begin{aligned}
			-\psi(x^{k},d^{k})&\geq\sum_{j=1}^{N_{k}}\left(-\frac{1}{N_{k}}\gamma_{k}\psi(x^{k},v(x^{k}))-\beta_{kj}\psi(x^{k},d^{k-j})\right)\\
			&=\sum_{j=1}^{N_{k}}\beta_{kj}(\gamma_{k}\phi_{kj}-\psi(x^{k},d^{k-j}))\\
			&\geq\sum_{j=1}^{N_{k}}\beta_{kj}\|JF(x^{k})\|\|d^{k-j}\|\\
			&\geq\frac{1}{2}\|v(x^{k})\|\sum_{j=1}^{N_{k}}\beta_{kj}\|d^{k-j}\|.
		\end{aligned}
	\end{equation}
	By \eqref{dual_sol}, there exists $\lambda(x^{k})\in\mathbb{R}^{m}$ with $\sum_{i=1}^{m}(\lambda(x^{k}))_{i}=1$ and $(\lambda(x^{k}))_{i}\geq0$ for each $i\in\langle m\rangle$ such that
	\begin{equation}\label{fk_dk_less1}
		v(x^{k})=-\sum_{i=1}^{m}(\lambda(x^{k}))_{i}\nabla F_{i}(x^{k}).
	\end{equation}
	From \eqref{fk_dk_less1} and \eqref{h} and observing that $\psi(x^{k},d^{k})<0$, we obtain
	\begin{equation*}
		\begin{aligned}[b]
			\langle v(x^{k}),d^{k}\rangle&=\sum_{i=1}^{m}(\lambda(x^{k}))_{i}(-\langle\nabla F_{i}(x^{k}),d^{k}\rangle)\\
			&\geq\sum_{i=1}^{m}(\lambda(x^{k}))_{i}(-\psi(x^{k},d^{k}))\\
			&=|\psi(x^{k},d^{k})|\sum_{i=1}^{m}(\lambda(x^{k}))_{i}\\
			&=|\psi(x^{k},d^{k})|.
		\end{aligned}
	\end{equation*}
	Observing $0<\varsigma^{2}\leq\|v(x^{k})\|^{2}\leq-2\psi(x^{k},v(x^{k}))\leq-2\psi(x^{k},d^{k})/\sigma$, where the last inequality follows from Lemma \ref{sufficient_descent}, we get  $|\psi(x^{k},d^{k})|\geq \sigma\varsigma^{2}/2>0$. From Remark \ref{vkdk}, if we take $\eta=2\xi_{1}\xi_{2}/ (\sigma\varsigma^{2})$, then
	\begin{equation}\label{fk_dk_less11}
		\langle v(x^{k}),d^{k}\rangle\leq \eta|\psi(x^{k},d^{k})|.
	\end{equation}
	By \eqref{dk}, we have 
	\begin{equation*}
		d^{k}-\gamma_{k} v(x^{k})=\sum_{j=1}^{N_{k}}\beta_{kj}d^{k-j}.
	\end{equation*}
	From the above relation and \eqref{fk_dk_less11}, it follows that
	\begin{equation}\label{convergence_analysis10}
		\begin{aligned}
			\|d^{k}\|^{2}&=-\gamma_{k}^{2}\|v(x^{k})\|^{2}+2\gamma_{k}\langle v(x^{k}),d^{k}\rangle+\left\|\sum_{j=1}^{N_{k}}\beta_{kj}d^{k-j}\right\|^{2}\\
			&\leq-\gamma_{k}^{2}\|v(x^{k})\|^{2}+2 \eta\gamma_{k}|\psi(x^{k},d^{k})|+\left\|\sum_{j=1}^{N_{k}}\beta_{kj}d^{k-j}\right\|^{2}.
		\end{aligned}
	\end{equation}
	Dividing both sides of \eqref{convergence_analysis10} by $\psi^{2}(x^{k},d^{k})$, we have
	\begin{equation*}
		\begin{aligned}
			\frac{\|d^{k}\|^{2}}{\psi^{2}(x^{k},d^{k})}&\leq\frac{\left\|\sum_{j=1}^{N_{k}}\beta_{kj}d^{k-j}\right\|^{2}}{\psi^{2}(x^{k},d^{k})}+\frac{2 \eta\gamma_{k}|\psi(x^{k},d^{k})|}{\psi^{2}(x^{k},d^{k})}-\frac{\gamma_{k}^{2}\|v(x^{k})\|^{2}}{\psi^{2}(x^{k},d^{k})}\\
			&=\frac{\left\|\sum_{j=1}^{N_{k}}\beta_{kj}d^{k-j}\right\|^{2}}{\psi^{2}(x^{k},d^{k})}+\frac{2\eta\gamma_{k}}{|\psi(x^{k},d^{k})|}-\frac{\gamma_{k}^{2}\|v(x^{k})\|^{2}}{|\psi(x^{k},d^{k})|^{2}}\\
			&=\frac{\left\|\sum_{j=1}^{N_{k}}\beta_{kj}d^{k-j}\right\|^{2}}{\psi^{2}(x^{k},d^{k})}-\left(\frac{\gamma_{k}\|v(x^{k})\|}{|\psi(x^{k},d^{k})|}-\frac{\eta}{\|v(x^{k})\|}\right)^{2}+\frac{\eta^{2}}{\|v(x^{k})\|^{2}}\\
			&\leq\frac{\left\|\sum_{j=1}^{N_{k}}\beta_{kj}d^{k-j}\right\|^{2}}{\psi^{2}(x^{k},d^{k})}+\frac{\eta^{2}}{\|v(x^{k})\|^{2}}\\
			&\leq\frac{4}{\|v(x^{k})\|^{2}}+\frac{\eta^{2}}{\|v(x^{k})\|^{2}}\\
			&=\frac{4+\eta^{2}}{\|v(x^{k})\|^{2}},
		\end{aligned}
	\end{equation*}
	where the last inequality holds because \eqref{convergence_analysis00}. Thus,
	\begin{equation*}
		\frac{\psi^{2}(x^{k},d^{k})}{\|d^{k}\|^{2}}\geq\frac{\|v(x^{k})\|^{2}}{4+\eta^{2}}.
	\end{equation*}
	Therefore,
	\begin{equation*}
		\sum_{k=0}^{\infty}\frac{\psi^{2}(x^{k},d^{k})}{\|d^{k}\|^{2}}\geq\sum_{k=0}^{\infty}\frac{\|v(x^{k})\|^{2}}{4+\eta^{2}}\geq\sum_{k=0}^{\infty}\frac{\varsigma^{2}}{4+\eta^{2}}=\infty,
	\end{equation*}
	which contradicts \eqref{armijio_vk1}.
\end{proof}

\begin{remark}\normalfont
	The above theorem states that if any $\gamma_{k}$ and $\psi_{kj}$ ($j\in\langle N_{k}\rangle$) satisfying the conditions of Lemma \ref{sufficient_descent} are selected, then we can obtain the global convergence of the proposed method.
\end{remark}

\section{Convergence rate}
In this section, our attention is focused on deriving the convergence rates of MMG-I/MMG-II. We first derive the rate of $1/\sqrt{k}$ for non-convex MOPs. Then, base on a new assumption, it is shown that MMG-I has linear convergence rate.

\begin{theorem}\label{convergence_rate1}
	Assume that the sequence $\{(x^{k},d^{k})\}$ is produced by MMG-I and that {\rm (A2)} and {\rm(A3)} hold. Then, either there is an infinite subset $\mathcal{N}\subseteq\{1,2,\ldots\}$ such that $$\lim_{k\in\mathcal{N},k\rightarrow\infty}\frac{\|v(x^{k})\|}{\|d^{k}\|}=0,$$ or MMG-I has a convergence rate of the order of $1/\sqrt{k}$.
\end{theorem}

\begin{proof}
	If $\{\|d^{k}\|/\|v(x^{k})\|\}$ is unbounded, then there is an infinite subset $\mathcal{N}\subseteq\{1,2,\ldots\}$ such that
	$$\lim_{k\in\mathcal{N},k\rightarrow\infty}\frac{\|d^{k}\|}{\|v(x^{k})\|}=\infty.$$ Thus,
	$\lim_{k\in\mathcal{N},k\rightarrow\infty}\|v(x^{k})\|/\|d^{k}\|=0$.

	If $\{\|d^{k}\|/\|v(x^{k})\|\}$ has a bound, then there is a positive constant $\vartheta$ such that
	\begin{equation}\label{convergence_rate1_0}
		\frac{\|d^{k}\|}{\|v(x^{k})\|}\leq\vartheta.
	\end{equation}
	Observe that, by \eqref{convergence_rate1_0}, $\psi(x^{k},v(x^{k}))+\|v(x^{k})\|^{2}/2<0$ and Lemma \ref{sufficient_descent}, we obtain
	\begin{equation}\label{convergence_rate1_2_1}
		\frac{ c^{2}}{4\vartheta^{2}}\|v(x^{k})\|^{2}\leq\frac{c^{2}\|v(x^{k})\|^{4}}{4\|d^{k}\|^{2}}\leq \dfrac{c^{2}\psi^{2}(x^{k},v(x^{k}))}{\|d^{k}\|^{2}}\leq\dfrac{\psi^{2}(x^{k},d^{k})}{\|d^{k}\|^{2}}.
	\end{equation}
	By \eqref{armijio_dk_bounds0} and \eqref{convergence_rate1_2_1}, for each $i\in\langle m\rangle$, one has
	\begin{equation}\label{convergence_rate1_3}
		F_{i}(x^{k})-F_{i}(x^{k+1})\geq\varepsilon\|v(x^{k})\|^{2},
	\end{equation}
	where $\varepsilon=\omega c^{2}/(4\vartheta^{2})$. From \eqref{convergence_rate1_3}, we have
	\begin{equation}\label{convergence_rate1_5}
		\begin{aligned}
			F_{i}(x^{0})-F_{i}(x^{k})&=\sum_{j=0}^{k-1}(F_{i}(x^{j})-F_{i}(x^{j+1}))\\
			&\geq\varepsilon\sum_{j=0}^{k-1}\|v(x^{j})\|^{2}\\
			&\geq k\varepsilon\min_{0\leq j\leq k-1}\|v(x^{j})\|^{2} ,\quad\forall i\in\langle m\rangle.
		\end{aligned}
	\end{equation}
	From (A3) and the continuity of $F$, there is $\bar{F}\in\mathbb{R}^{m}$ such that $\bar{F}\preceq F(x^{k})$ for any $k$. Let $F^{\min}=\min_{i\in\langle m\rangle}\bar{F}_{i}$ and $F_{0}^{\max}=\max_{i\in\langle m\rangle}F_{i}(x^{0})$, where $x^{0}$ is a given initial point. Then, \eqref{convergence_rate1_5} implies that
	\begin{equation}\label{convergence_rate1_6}
		\min_{0\leq j\leq k-1}\|v(x^{j})\|\leq \sqrt{\frac{F^{\max}_{0}-F^{\min}}{\varepsilon}}\frac{1}{\sqrt{k}}.
	\end{equation}
	Hence, the proof is complete.
\end{proof}

\begin{remark}\normalfont
	Similar to Theorem \ref{convergence_rate1} that MMG-II has the same convergence rate by taking $\varepsilon=c^{2}/(16L\vartheta^{2})$ in \eqref{convergence_rate1_3}.
\end{remark}

\begin{remark}\normalfont
	We would like to mention that the convergence rate of our method is consistent with that of  multiobjective steepest descent method \cite{F_s2000} when $F$ is non-convex (see \cite[Theorem 3.1]{F_c2019}).
\end{remark}

To derive the linear convergence of MMG-I, the following hypothesis concerning the function $F$ is required.
\begin{description}\setlength{\itemsep}{-0.05in}
	\item[(A4)] $F$ satisfies the following relation, i.e.,
	\begin{equation}\label{pl}
		\kappa(F(x)-F(\bar{x}))\preceq \frac{1}{2}\|JF(x)\|^{2}e,
	\end{equation}
	for all $x\in\mathbb{R}^{n}$ and for some $\kappa>0$, where $\bar{x}$ is Pareto critical.
\end{description}

We give a simple example to illustrate (A4). 

\begin{example}\normalfont
	The multiobjective optimization problem is given as follows:
	\begin{equation*}
		\min_{x\in\mathbb{R}}\quad F(x)=(x^{2}-4,(x-1)^{2})^{\top}, 
	\end{equation*}
	which was given in \cite{A_a2015} and used as a test function in \cite{G_g2021,G_a2021}. Clearly, $F$ is convex and the set of Pareto critical points is $[0,1]$. It is easy to verify that $F$ satisfies (A4) with $\kappa=1/2$.
\end{example}

\begin{theorem}\label{conv_rate}
	Assume that the sequence $\{(x^{k},d^{k})\}$ is produced by MMG-I and that $\{x^{k}\}$ converges to $\bar{x}$, where $\bar{x}$ is Pareto critical. If {\rm (A2)}--{\rm(A4)} hold, then either there is an infinite subset $\mathcal{K}\subseteq\{1,2,\ldots\}$ such that
	\begin{equation}\label{conv_rate0}
		\lim_{k\in\mathcal{K},k\rightarrow\infty}\frac{\|v(x^{k})\|}{\|d^{k}\|}=0,
	\end{equation}
	or there is $\mu\in(0,1)$ such that
	\begin{equation*}
		F(x^{k})-F(\bar{x})\preceq(1-\mu)^{k}(F(x^{0})-F(\bar{x})).
	\end{equation*}
\end{theorem}

\begin{proof}
	If $\{\|d^{k}\|/\|v(x^{k})\|\}$ has no bound, then \eqref{conv_rate0} holds by the proof of Theorem \ref{convergence_rate1}. If $\{\|d^{k}\|/\|v(x^{k})\|\}$ has a bound, then from the proof of Theorem \ref{convergence_rate1}, it follows that
	\begin{equation}\label{conv_rate1}
		F_{i}(x^{k})-F_{i}(x^{k+1})\geq\varepsilon\|v(x^{k})\|^{2},\quad\forall i\in\langle m\rangle,
	\end{equation}
	where $\varepsilon=\omega \sigma^{2}/(4\vartheta^{2})$. From the continuity  of $JF$, the continuity of $v$ and (A3), there is a constant $\xi>\max\{1,\kappa\}$ such that $\|JF(x^{k})\|\leq \xi\|v(x^{k})\|$. This, combined with \eqref{conv_rate1}, gives 
	\begin{equation}\label{convergence_rate_2_1}
		F_{i}(x^{k})-F_{i}(x^{k+1})\geq\frac{\varepsilon}{2\xi^{3}}\|JF(x^{k})\|^{2}\geq\frac{\varepsilon \kappa }{\xi^{3}}(F_{i}(x^{k})-F_{i}(x^{*})),\quad\forall i\in\langle m\rangle.
	\end{equation}
	Let $\mu=\varepsilon \kappa/\xi^{3}$. Clearly, $\omega<1$ in view of \eqref{ww}. Moreover, $\kappa/\xi< 1$ and
	\begin{equation}
		1\geq\frac{-\psi(x^{k},d^{k})}{\|JF(x^{k})\|\|d^{k}\|}\geq\dfrac{-c\psi(x^{k},v(x^{k}))}{\|JF(x^{k})\|\|d^{k}\|}\geq\dfrac{\sigma\|v(x^{k})\|^{2}}{2\|JF(x^{k})\|\|d^{k}\|}\geq\dfrac{\sigma}{2\xi\vartheta},
	\end{equation}
	where the last inequality follows from \eqref{convergence_rate1_0}. Therefore, we conclude  $\mu\in(0,1)$. Adding $F_{i}(\bar{x})$ from both sides of \eqref{convergence_rate_2_1} gives
	\begin{equation*}
		F_{i}(x^{k+1})-F_{i}(\bar{x})\leq(1-\mu)(F_{i}(x^{k})-F_{i}(\bar{x})),\quad\forall i\in\langle m\rangle.
	\end{equation*}
	Applying the above inequality recursively gives
	\begin{equation*}
		F_{i}(x^{k+1})-F_{i}(\bar{x})\leq(1-\mu)^{k+1}(F_{i}(x^{0})-F_{i}(\bar{x})),\quad\forall i\in\langle m\rangle.
	\end{equation*}
	Therefore, the proof is complete.
\end{proof}

\section{Numerical experiments}
This section conducts some computational experiments to show the efficiency of the proposed method for solving MOPs. We just compare MMG-I with the multiobjective versions of the steepest descent (SD) algorithm proposed in \cite{F_s2000} and the Fletcher--Reeves (FR), conjugate descent (CD), and Hestenes--Stiefel (HS) nonlinear conjugate gradient methods proposed in \cite{L_n2018}. The codes are edited in $\mathsf{Python}$ programming language and run on a computer with CPU Intel Core i7 2.90GHz and 32GB of memory. In order to keep the fairness process in numerical implementation, the Armijio line search rule considered in all compared algorithms was implemented as a simple backtracking procedure, i.e., find a stepsize $\alpha_{k}=\max\{(1/2)^{i}:i=0,1,2,\ldots\}$ such that
\begin{equation*}
	F(x^{k}+\alpha_{k} d^{k})\preceq F(x^{k})+\rho\alpha_{k} \psi(x^{k},d^{k})e,
\end{equation*}
where $\rho=0.0001$. As illustrated in Lemma \ref{sufficient_descent}, we can select any parameters satisfying the conditions of Lemma \ref{sufficient_descent}. In MMG-I, we set $\gamma^{*}=10^{-10}$ and give two choices of $\gamma_{k}$ as follows
\begin{enumerate}[(i)]\setlength{\itemsep}{-0.05in}
	\item $\gamma_{k}=1$ for all $k$;
	\item $\gamma_{0}=1$ and for $k\geq1$
	\begin{equation*}
		\gamma_{k}=\left\{
		\begin{array}{lllll}
			1, &\quad{\rm if}\quad \frac{\|q^{k-1}\|}{\|h^{k-1}\|}<\gamma^{*},\\
			\frac{\|q^{k-1}\|}{\|h^{k-1}\|},& \quad{\rm otherwise},
		\end{array}\right.
	\end{equation*}
	where $q^{k-1}=x^{k}-x^{k-1}$, $h^{k-1}=v(x^{k})-v(x^{k-1})$.
\end{enumerate}
Clearly, the selection of $\gamma_{k}$ satisfies the condition Lemma \ref{sufficient_descent}. For a given $\gamma_{k}$, we set $\psi_{kj}$ as
$$\phi_{kj}=\frac{\psi(x^{k},d^{k-j})+\|JF(x^{k})\|\|d^{k-j}\|+\zeta}{\gamma_{k}},$$
where $\zeta>0$. Note that the choice of $\phi_{kj}$ satisfies condition \eqref{sufficient_descent0001}. We mark MMG-I whose $\gamma_{k}$ satisfies (i) and (ii) as MMG-I1 and MMG-I2, respectively.

At each iteration, we solve the dual problem \eqref{dual} with $x=x^{k}$ using the \texttt{minimize} solver of $\mathsf{Python}$. All runs were stopped at $x^{k}$ whenever $|\theta(x^{k})|\leq 10^{-6}$. Since Proposition \ref{pa_sta_equ} means that $v(x)=0$ iff $\theta(x)=0$, this stopping criteria is desirable. Moreover, we set the maximum number of external iterations to 10000.

The test problems listed in Table \ref{problems} were selected from different literature for multiobjective optimization. The first and second columns in Table \ref{problems} are the name of test problems and the related reference, respectively. Column ``$n$'' stands for the number of variables of these problems. Column ``Convex'' states the convexity of the test problem. The initial points $x^{0}$ are generated inside a box $[x_{L},x_{U}]$, where $x_{L}$ in column 5 stand for the lower bounds of variables and $x_{U}$ in column 6 is the upper bounds of variables. For every algorithm, these problems were solved 200 times employing initial points uniformly and randomly distributed in the corresponding box.

\begin{table}[H]
	\setlength{\abovecaptionskip}{0cm}  
	\setlength{\belowcaptionskip}{-0.2cm} 
	\centering
	\caption{The relevant information of test instances.}
	\label{problems}
	\begin{threeparttable}\scriptsize
		\begin{tabular}{llllll}
			\hline
			\rowcolor{gray!20}Problem & Source  & $n$ & Convex & $x_{L}$   & $x_{U}$       \\ \hline
			AP3     &    \cite{A_a2015}     & 2  & Y      & $(-100,-100,)$ &  $(100,100)$\\ 
			SK2     &    \cite{H_a2006}    & 4 & N      & $(-10,-10,-10,-10)$ &  $(10,10,10,10)$\\ 
			DD1$^{*}$     &    \cite{D_n1998}     & 5 & N      & $(-20,\ldots,-20)$ &  $(20,\ldots,20)$\\ 
			DGO1     &     \cite{H_a2006}    & 1 & N      & $-10$ & 13\\ 
			DGO2     &    \cite{H_a2006}     & 1 & Y      &$ -9$ & 9 \\ 
			Toi4$^{\dagger}$     & \cite{T_t1983}       & 4 & Y      & $(-2,\ldots,-2)$ &  $(5,\ldots,5)$\\
			Far1     &     \cite{H_a2006}    & 2 & N      & $(-1,-1)$ &  $(1,1)$\\
			BK1	& 	\cite{H_a2006}		&2 &Y			& $(-5,-5)$ &  $(10,10)$\\
			LE1	& 	\cite{H_a2006}		&2 &N			& $(-5,-5)$ &  $(10,10)$\\
			SLC2     &  \cite{S_t2011}       & 10 & Y      & $(-3,\ldots,-3)$ &  $(3,\ldots,3)$\\ 
			SD     &   \cite{S_m1993}      & 4 & Y      & $(1,\sqrt{2},\sqrt{2},1)$ & $(3,3,3,3)$\\ 
			MOP2     &     \cite{H_a2006}      & 2 & N      & $(-4,-4)$ &  $(4,4)$\\
			MOP3     &     \cite{H_a2006}     & 2  & N      & $(-\pi,-\pi)$ & $(\pi,\pi)$\\ 
			PNR     &    \cite{P_p2006}     & 2 & Y      & $(-1,-1)$ &  $(1,1)$\\
			VU1     &      \cite{H_a2006}     & 2 & N      & $(-3,-3)$ &  $(3,3)$ \\ 
			KW2     &   \cite{K_a2005}      & 2 & N      & $(-3,-3)$ &  $(3,3)$ \\ 
			MMR1$^{\star}$     &    \cite{M_b2008}     & 2 & N      & $(0.1,0)$ &  $(1,1)$ \\ 
			MMR3 & 	\cite{M_b2008}		 & 2 & N    & $(-1,-1)$ &  $(1,1)$ \\ 
			Lov3 &      \cite{L_s2011}   &2 & N      & $(-20,-20)$ &  $(20,20)$\\ 
			Lov4 &      \cite{L_s2011}   &2 & N      & $(-20,-20)$ &  $(20,20)$\\ 
			Lov6 &    \cite{L_s2011}     &6 & N      & $(-0.1,-0.16,\ldots,-0,16)$ & $(0.425,0.16,\ldots,0,16)$ \\ 
			FF1     &      \cite{H_a2006}     & 2 & N      & $(-1,-1)$ &  $(1,1)$\\
			FF1$^{\ddagger}$a     &     \cite{H_a2006}      & 100 & N     & $(-10,\ldots,-10)$ &  $(10,\ldots,10)$\\ 
			FF1$^{\ddagger}$b     &     \cite{H_a2006}      & 200 & N     & $(-10,\ldots,-10)$ &  $(10,\ldots,10)$\\ 
			JOS1a     &    \cite{J_d2001}     & 50 & Y      & $(-100,\ldots,-100)$ &  $(100,\ldots,100)$\\ 
			JOS1b     &    \cite{J_d2001}     & 100 & Y     & $(-100,\ldots,-100)$ &  $(100,\ldots,100)$\\ 
			JOS1c     &    \cite{J_d2001}     & 200 & Y      & $(-100,\ldots,-100)$ &  $(100,\ldots,100)$\\ 
			JOS1d     &    \cite{J_d2001}     & 500 & Y      & $(-100,\ldots,-100)$ &  $(100,\ldots,100)$\\ 
			\hline
		\end{tabular}
		\begin{tablenotes}
			\item[*] It is a revamped form of DD1, which can be obtained in \cite{M_n2019}.
			\item[$\star$] It is a modified version of MMR1. We set $F_{1}(x)=1+x_{1}^{2}$ and $F_{2}(x_{2})=h(x)/F_{1}(x)$, where $h(x_{2})$ is defined in \cite{M_b2008}.
			\item[$\dagger$] It is an extension of a scalar optimization problem, which is presented in \cite{M_n2019}.
			\item[$\ddagger$] It is a modified form of FF1 that presented in \cite{E_a2021}.
		\end{tablenotes}
	\end{threeparttable}
\end{table}

Taking into account of numerical reasons, we use the following scaled version of problem \eqref{mop} which was also considered in \cite{G_a2021, G_g2021}:
\begin{equation}\label{mop1}
	\min_{x\in\mathbb{R}^{n}}\quad (r_{1} F_{1}(x),r_{2} F_{2}(x),...,r_{m}F_{m}(x))^{\top},
\end{equation}
where the scaling factors $r_{i}=1/\max\{1,\|\nabla F_{i}(x^{0})\|_{\infty}\}$, $i\in\langle m\rangle$, $x^{0}$ is the given initial point. As illustrated in \cite{G_a2021, G_g2021}, problem \eqref{mop} is equivalent to problem \eqref{mop1}. In other words, the same Pareto critical points can be obtained by solving them.

Next, we adopt the performance profile proposed in \cite{D_b2002} by Dolan and Mor\'{e} to analyze the performance of various algorithms. The performance profile has become a commonly-used tool to estimate the performance of multiple solvers $\mathcal{S}$ when run on a test set $\mathcal{P}$ in scalar optimization. It is worth mentioning that such tool is also used in multiobjective optimization; see \cite{M_n2019,M_b2016, G_a2021, G_g2021,C_d2011}. We briefly describe the tool in the sequel. Suppose that there exist $n_{s}$ solvers and $n_{p}$ problems. For $s\in\mathcal{S}$ and $p\in\mathcal{P}$, we denote $o_{p,s}$ by the performance of $s$ on the problem $p$. The performance ratio is $z_{p,s}=o_{p,s}/\min\{o_{p,s}:s\in\mathcal{S}\}$ and the cumulative distribution function $\rho_{s}:[1,\infty)\rightarrow[0,1]$ is
$$\rho_{s}(\tau)=\frac{\left|\left\{p\in\mathcal{P}:z_{p,s}\leq\tau\right\}\right|}{n_{p}}.$$
Therefore, the performance profile is presented by depicting the figure of the cumulative distribution function $\rho_{s}$. Note that $\rho_{s}(1)$ refers to the probability that the solver defeat the remaining solvers. The right of the image for $\rho_{s}$ shows the robustness associated with a solver.


\subsection{Analysis of $\gamma_{k}$ and $N$.}

We first discuss the impact of the choice of $\gamma_{k}$ on the algorithmic performance. In other words, we compare MMG-I1 and MMG-I2 by virtue of the number of iterations and function evaluations under the same $N$ value. We would like to emphasize specifically that the results for the number of iterations and gradient evaluations are similar from our computational experiments, so the number of gradient evaluations is ignored in our later analysis. We test them with the values $N=1, 3, 5, 7, 9$. Figures \ref{pk_iter} and \ref{pk_evalf} respectively present the performance profiles from the viewpoint of iterations and function evaluations. As can be seen, the MMG-I algorithm is vulnerable to parameter $\gamma_{k}$ under the same $N$.

\begin{figure}[H]
	\centering  
	\subfigure[$N=1	$]{
		\label{pk1}
		\includegraphics[width=0.3\textwidth]{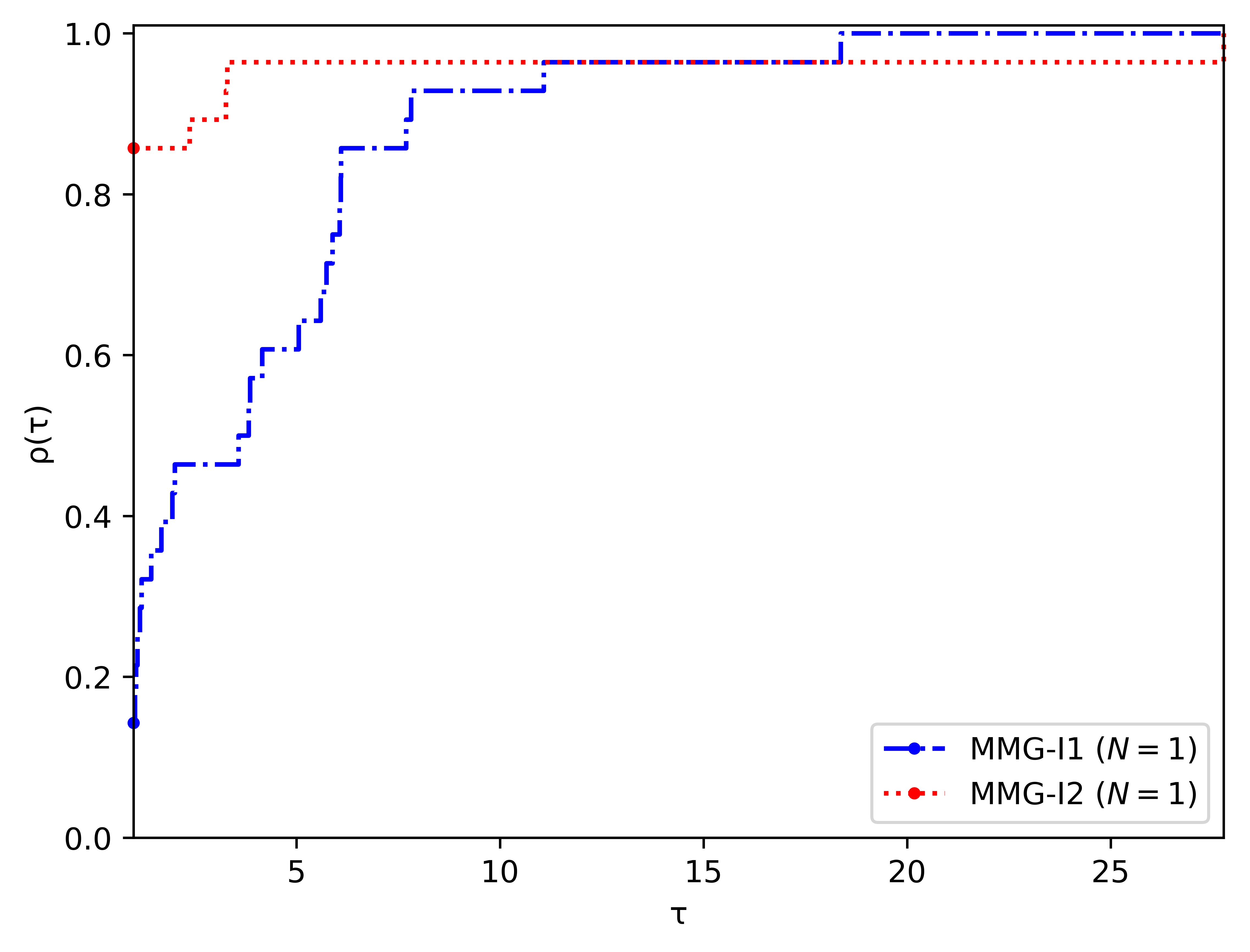}}
	\subfigure[$N=3	$]{
		\label{pk2}
		\includegraphics[width=0.3\textwidth]{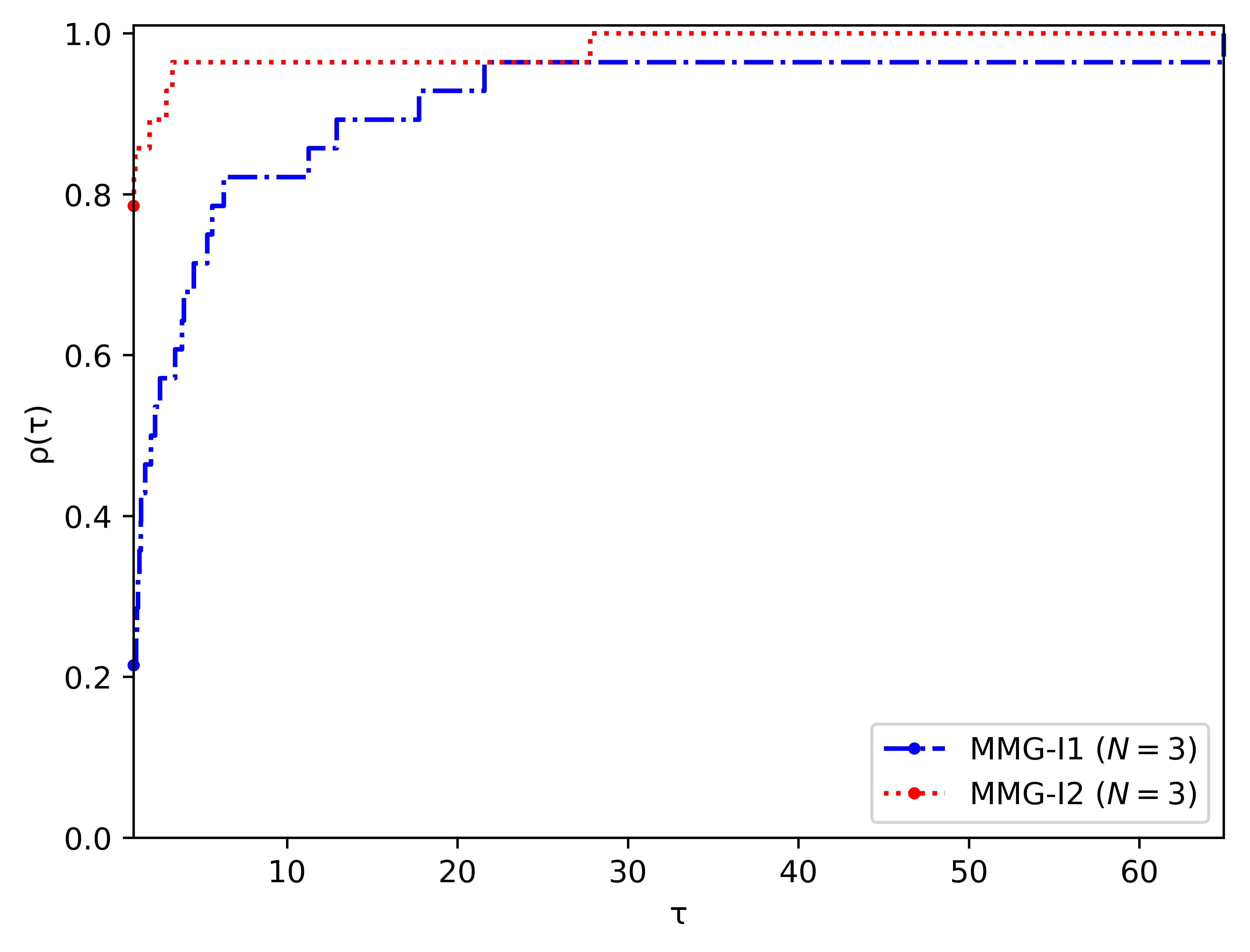}}
	\subfigure[$N=5	$]{
		\label{pk3}
		\includegraphics[width=0.3\textwidth]{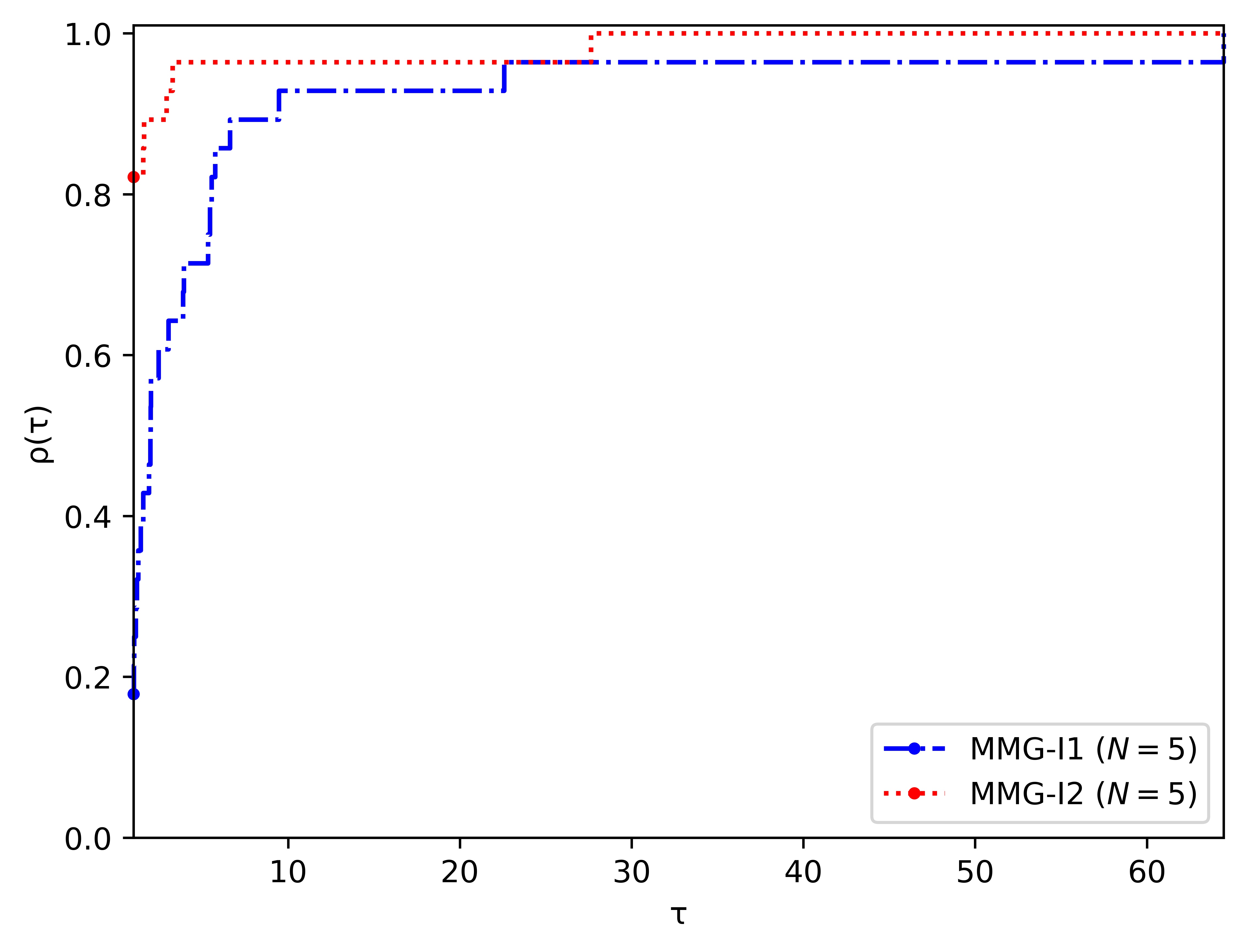}}
	\subfigure[$N=7	$]{
		\label{pk4}
		\includegraphics[width=0.3\textwidth]{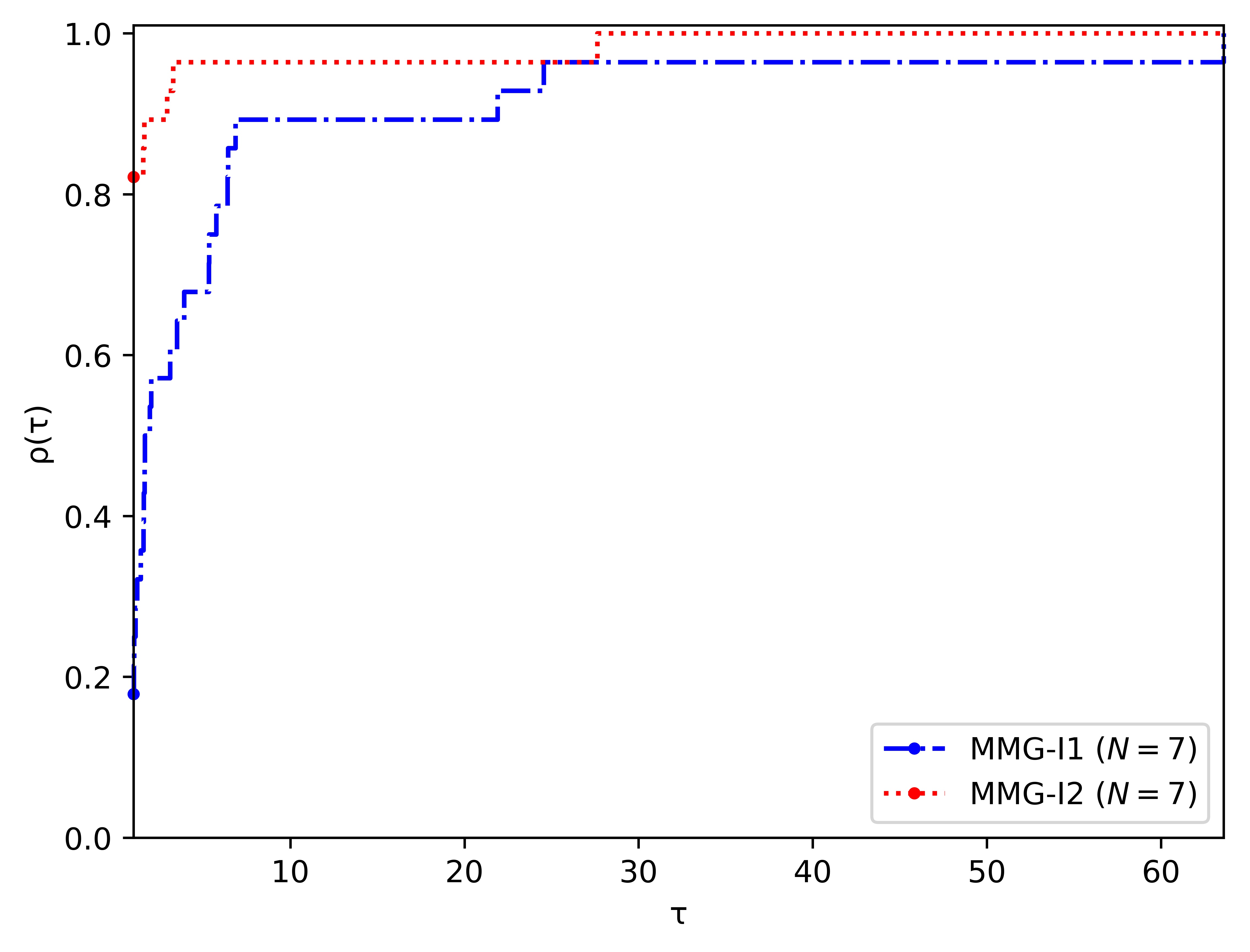}}
	\subfigure[$N=9	$]{
		\label{pk5}
		\includegraphics[width=0.3\textwidth]{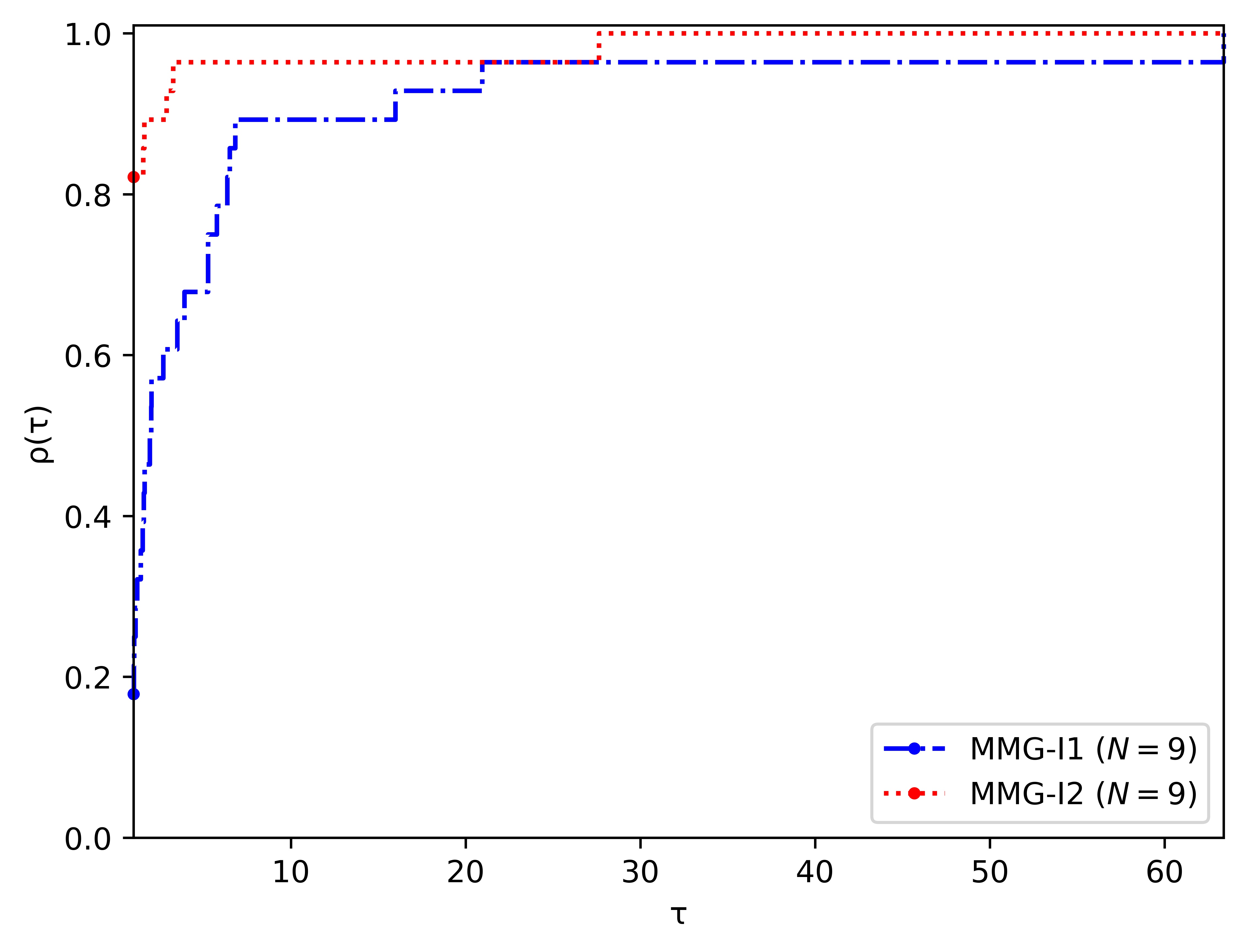}}
	\caption{Performance profiles comparing the number of iterations of MMG-I1 and MMG-I2 under the same $N$.}
	\label{pk_iter}
\end{figure}

\begin{figure}[H]
	\centering  
	\subfigure[$N=1	$]{
		\label{pk11}
		\includegraphics[width=0.3\textwidth]{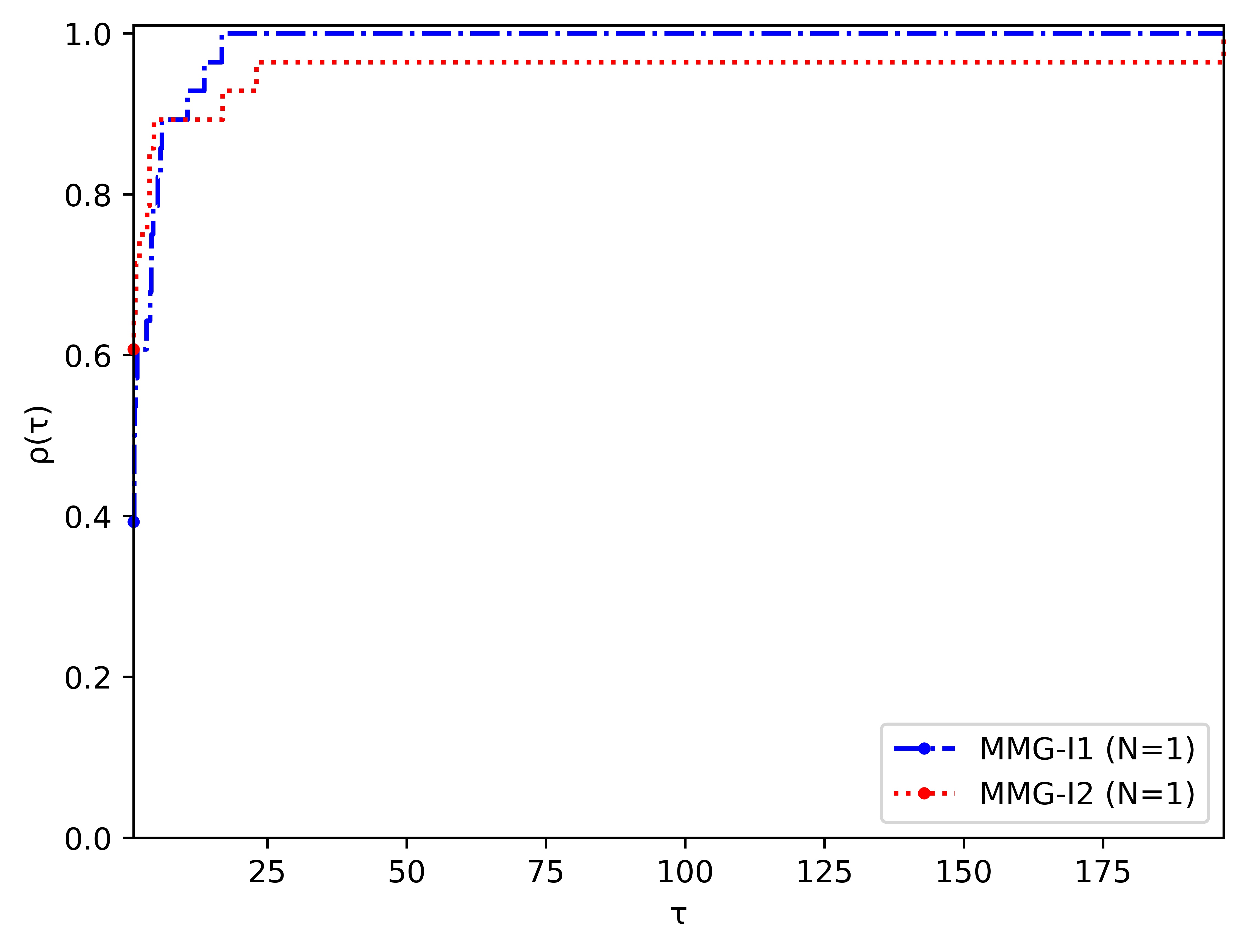}}
	\subfigure[$N=3	$]{
		\label{pk21}
		\includegraphics[width=0.3\textwidth]{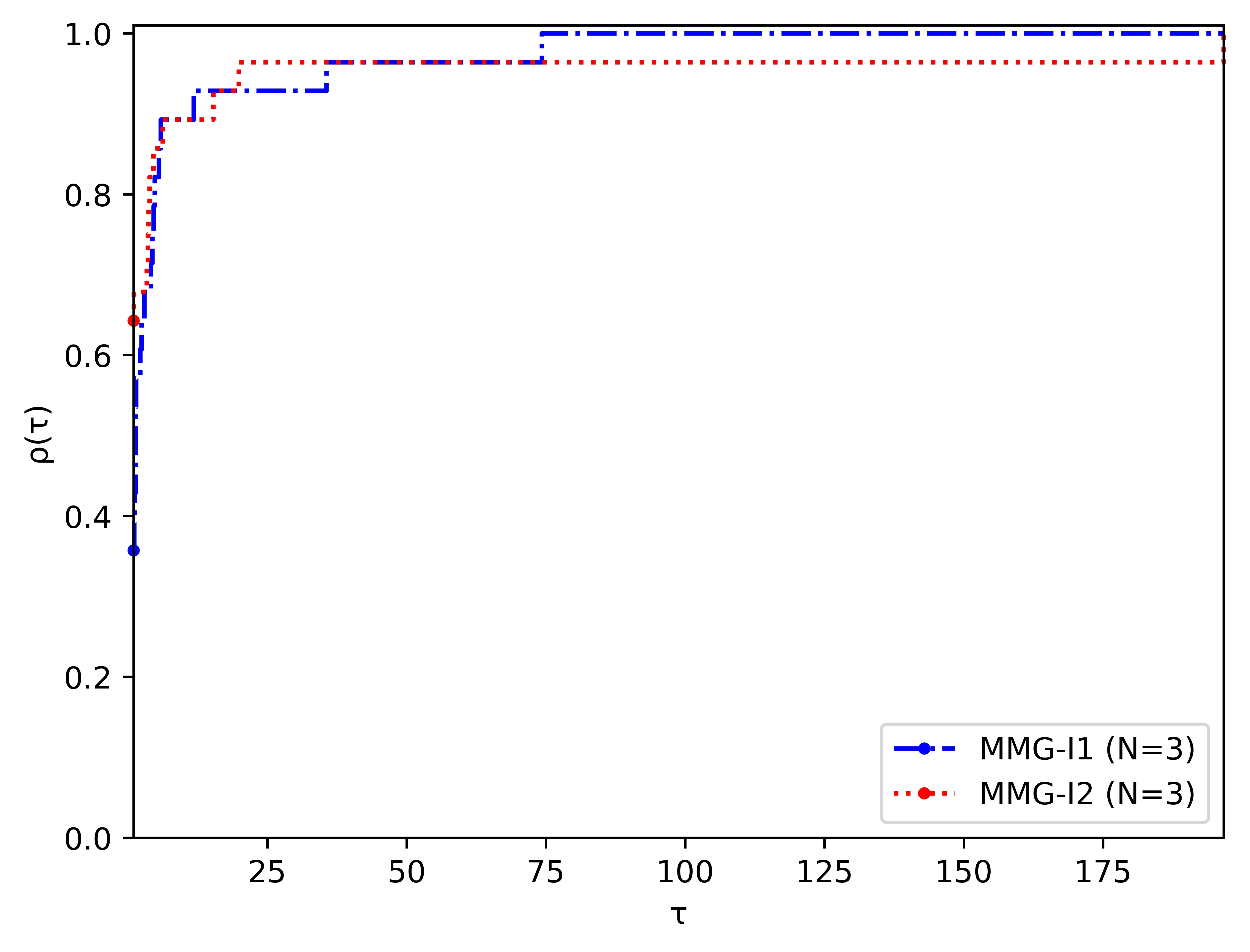}}
	\subfigure[$N=5	$]{
		\label{pk31}
		\includegraphics[width=0.3\textwidth]{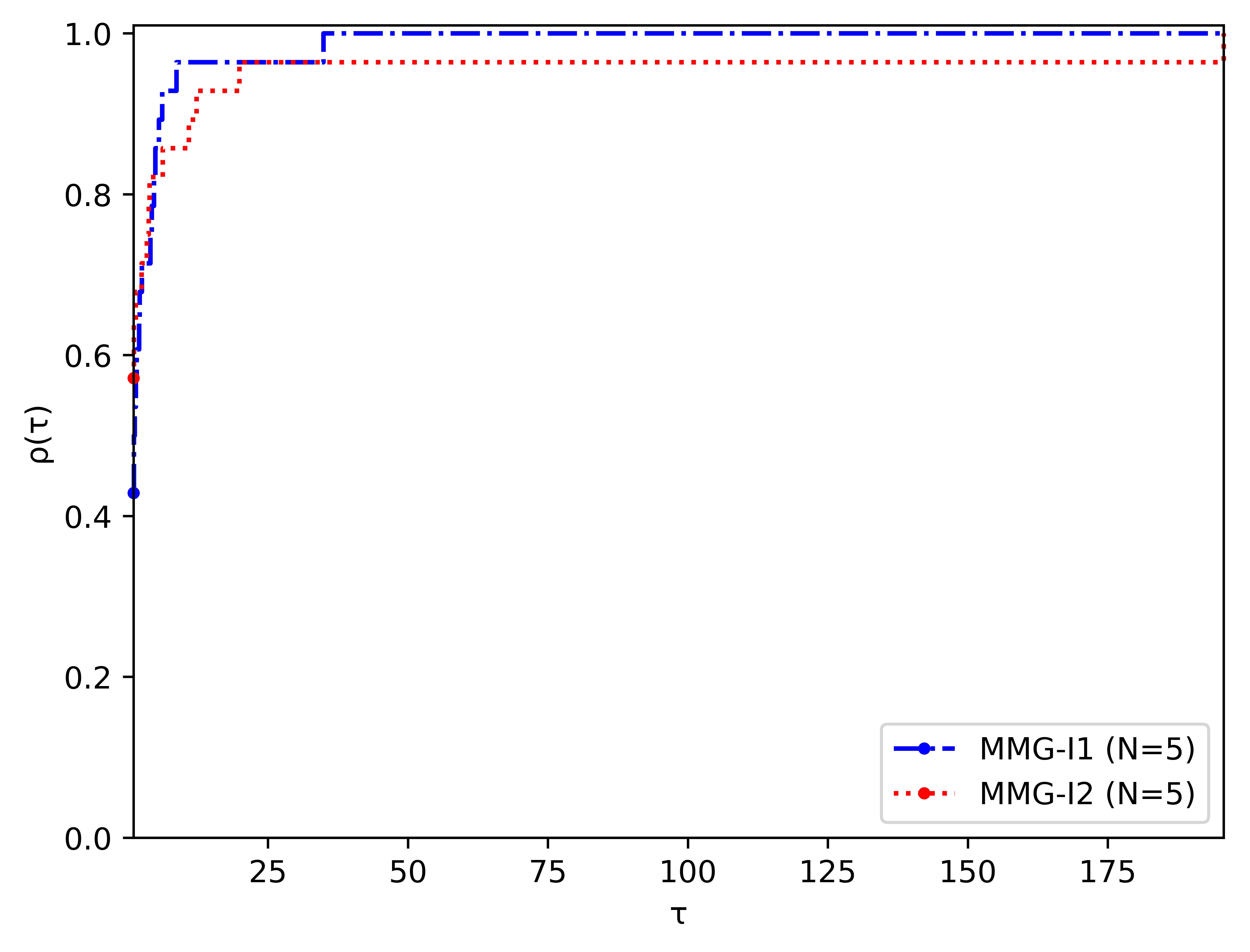}}
	\subfigure[$N=7	$]{
		\label{pk41}
		\includegraphics[width=0.3\textwidth]{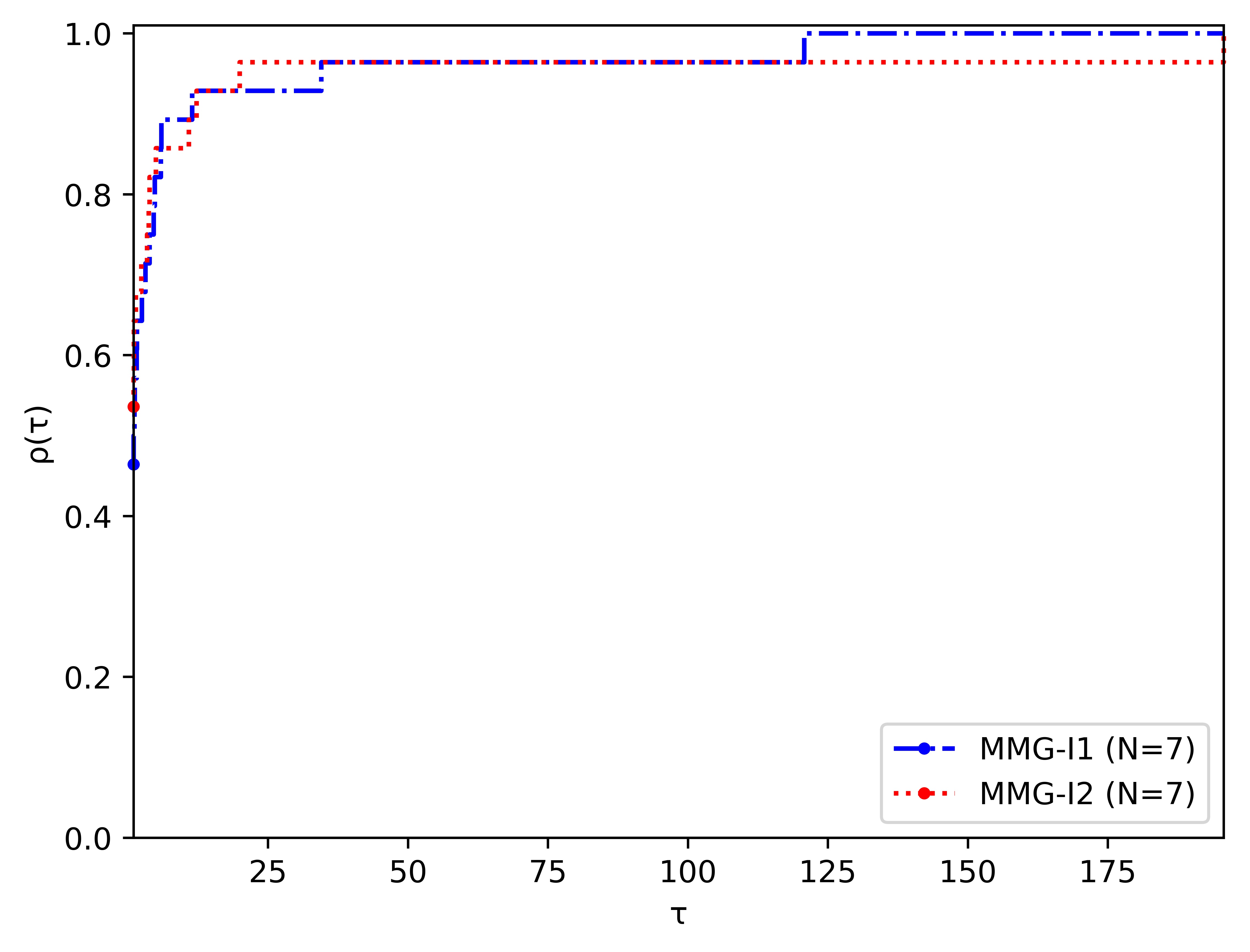}}
	\subfigure[$N=9	$]{
		\label{pk51}
		\includegraphics[width=0.3\textwidth]{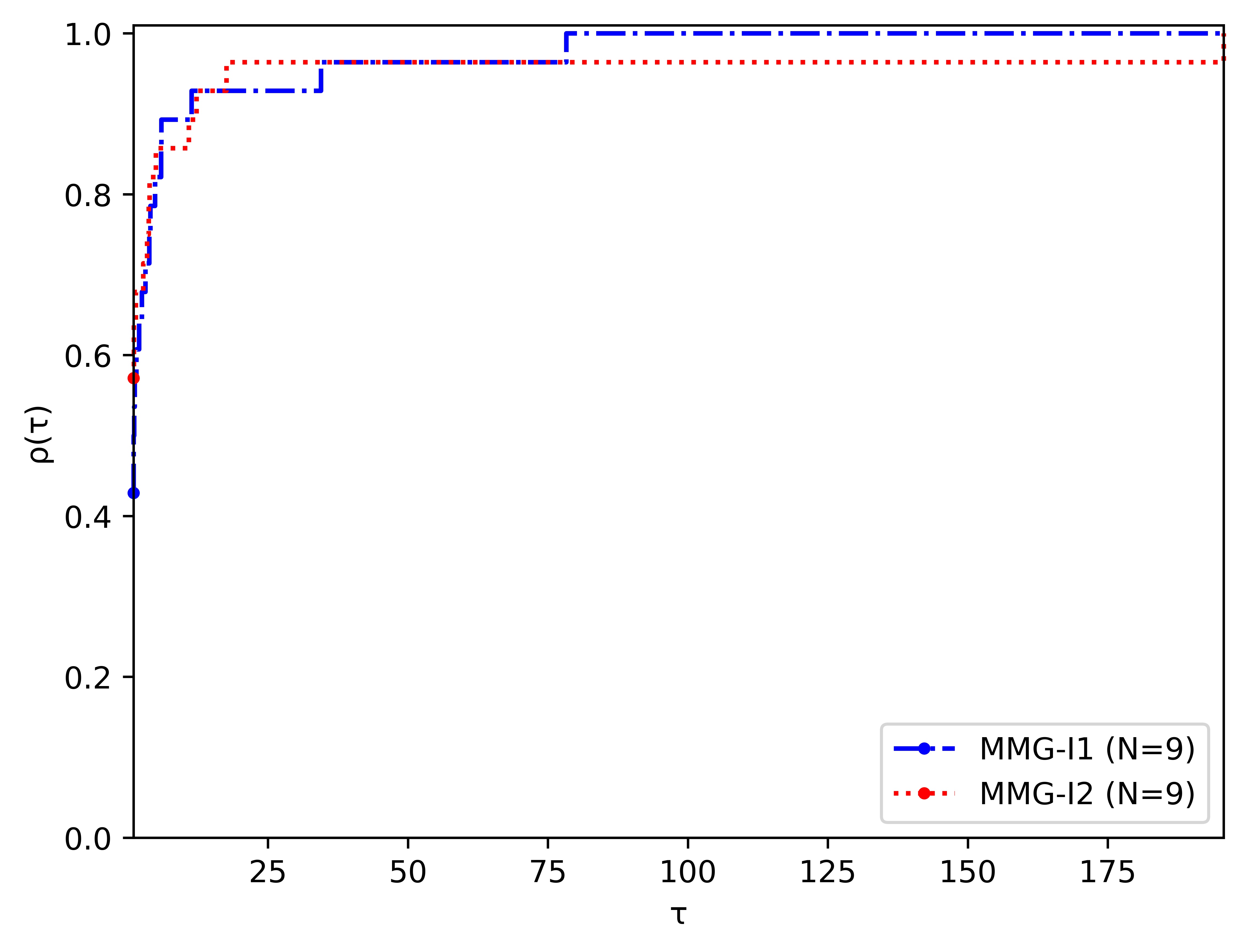}}
	\caption{Performance profiles comparing the number of function evaluations of MMG-I1 and MMG-I2 under the same $N$.}
	\label{pk_evalf}
\end{figure}

To study how a selection of the number of $N$ in MMG-I1 and MMG-I2 affects numerical performance, we respectively test MMG-I1 and MMG-I2 with the values $N=1, 3, 5, 7, 9$. Figures \ref{MMGI1} and \ref{MMGI2} respectively show the performance profiles in terms of iterations and function evaluations for MMG-I1 and MMG-I2. As can be seen, MMG-I1 or MMG-I2 with various $N$ values has different performance evaluations in different ranges. Therefore, the performance of MMG-I1 and MMG-I2 rely on the choice of parameter $N$. Although it is difficult for us to present the best choice from a theoretical point of view, the choice $N=5$ for MMG-I1 and the choice $N=3$ for MMG-I2 are relatively robust in our experiments.

\begin{figure}[H]
	\centering  
	\subfigure[Iterations]{
		\label{Fig.sub.11}
		\includegraphics[width=0.45\textwidth]{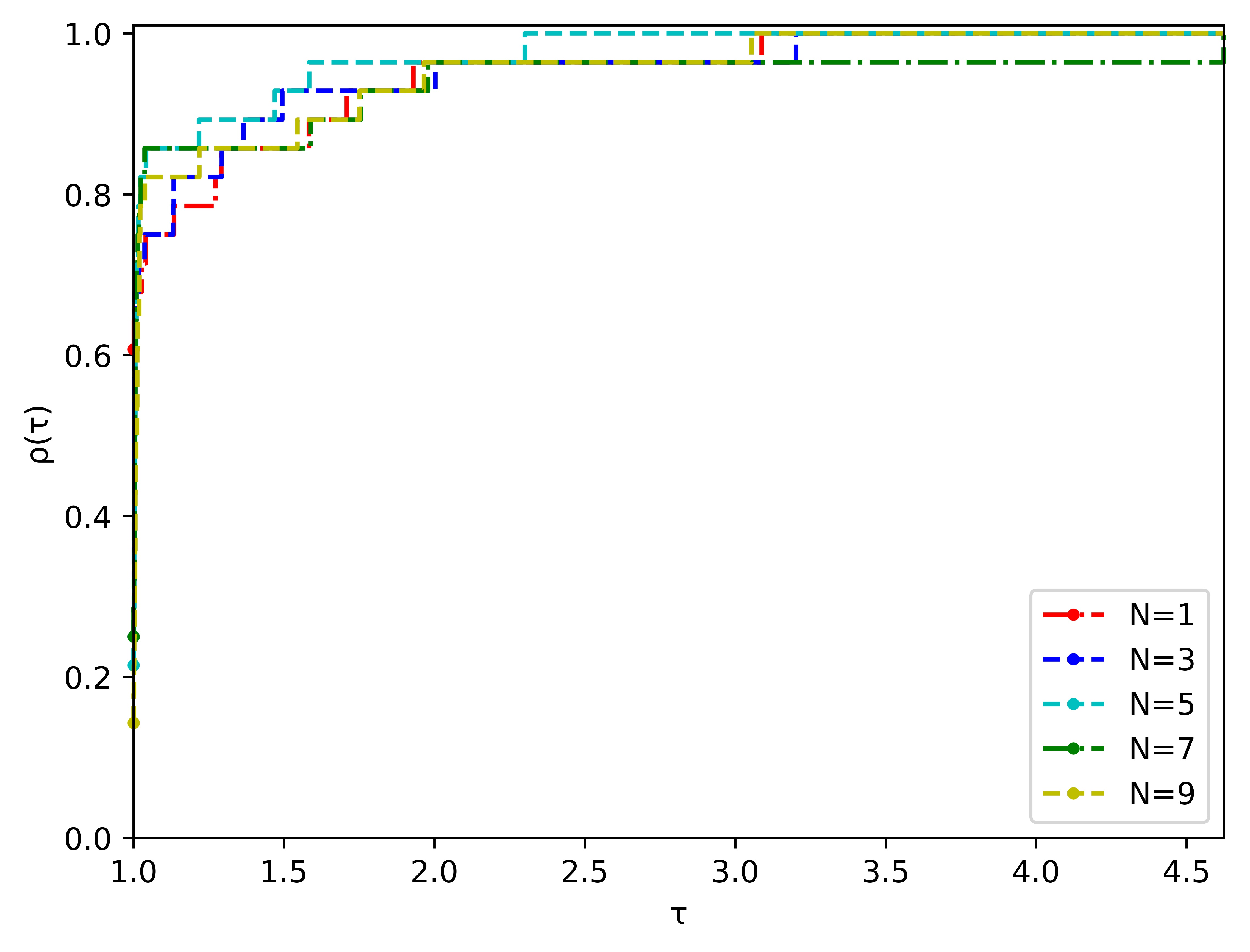}}
	\subfigure[Function evaluations]{
		\label{Fig.sub.12}
		\includegraphics[width=0.45\textwidth]{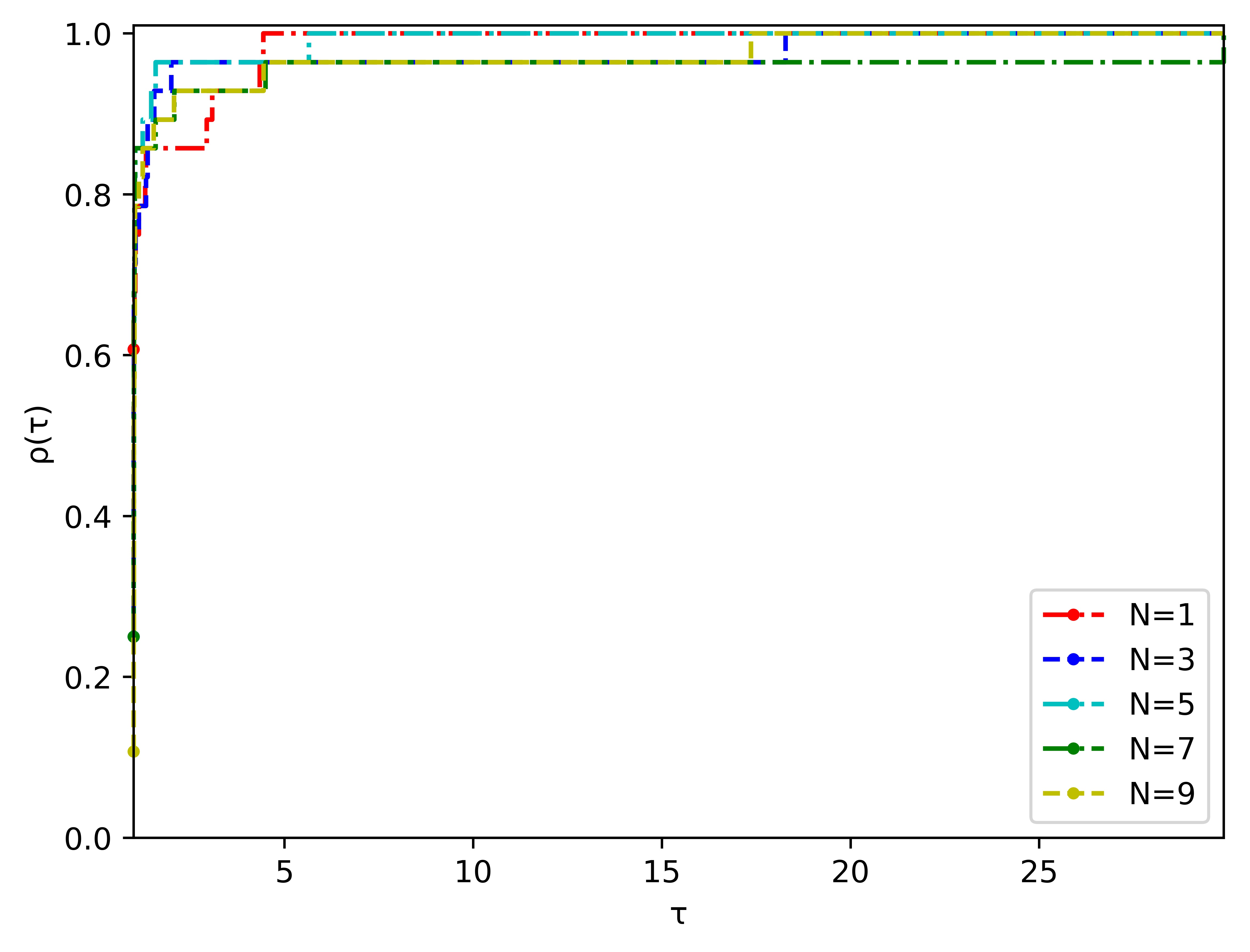}}
	\caption{Performance profiles for MMG-I1 with different $N$.}
	\label{MMGI1}
\end{figure}

\begin{figure}[H]
	\centering  
	\subfigure[Iterations]{
		\label{Fig.sub.1}
		\includegraphics[width=0.45\textwidth]{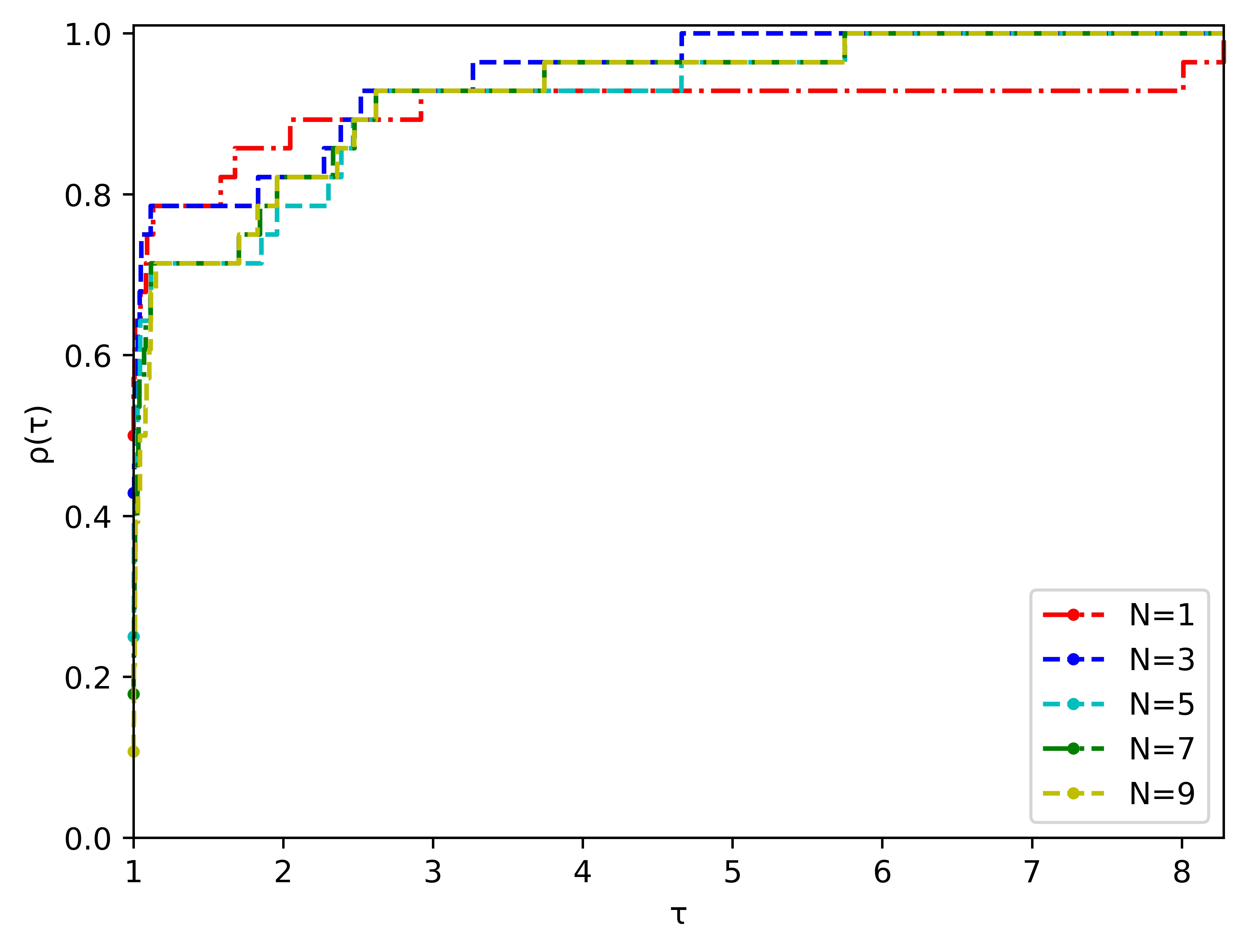}}
	\subfigure[Function evaluations]{
		\label{Fig.sub.2}
		\includegraphics[width=0.45\textwidth]{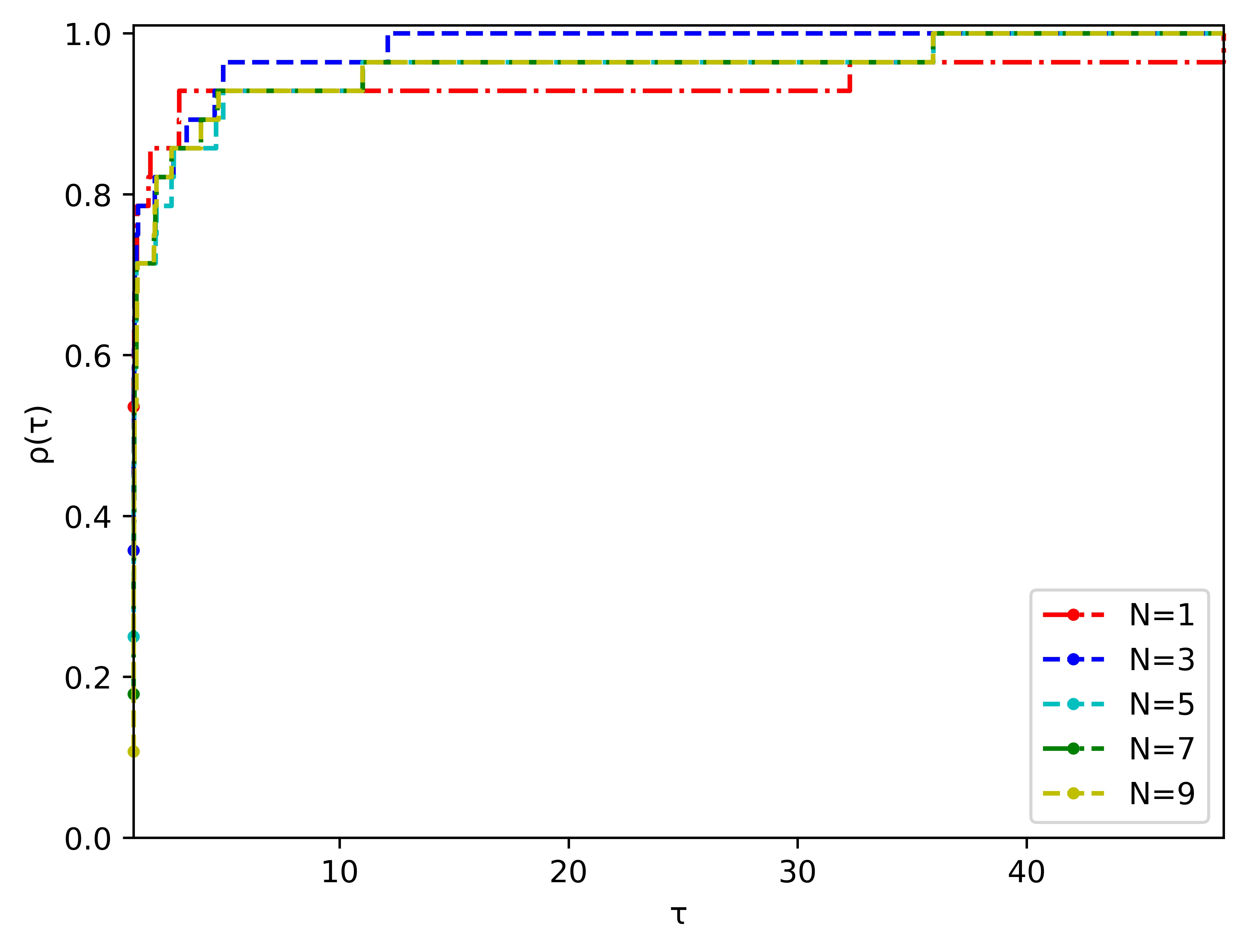}}
	\caption{Performance profiles for MMG-I2 with different $N$.}
	\label{MMGI2}
\end{figure}

\subsection{Analysis of all compared algorithms}\label{subsec.6.2}

In what follows, we will present some performance results of SD, FR, CD, HS, MMG-I1 with $N=5$ and MMG-I2 with $N=3$. Figures \ref{all_iter_tau=None} and \ref{all_evalf_tau=None} respectively give the performance profiles of the number of iterations and function evaluations for the six algorithms. We also plot the performance profiles in smaller intervals so that the difference between them becomes more obvious (see Figures \ref{all_iter_tau=30} and \ref{all_evalf_tau=200}). 

\begin{figure}[H]
	\centering  
	\subfigure[Proformance profile {\rm [0,350]}.]{
		\label{all_iter_tau=None}
		\includegraphics[width=0.45\textwidth]{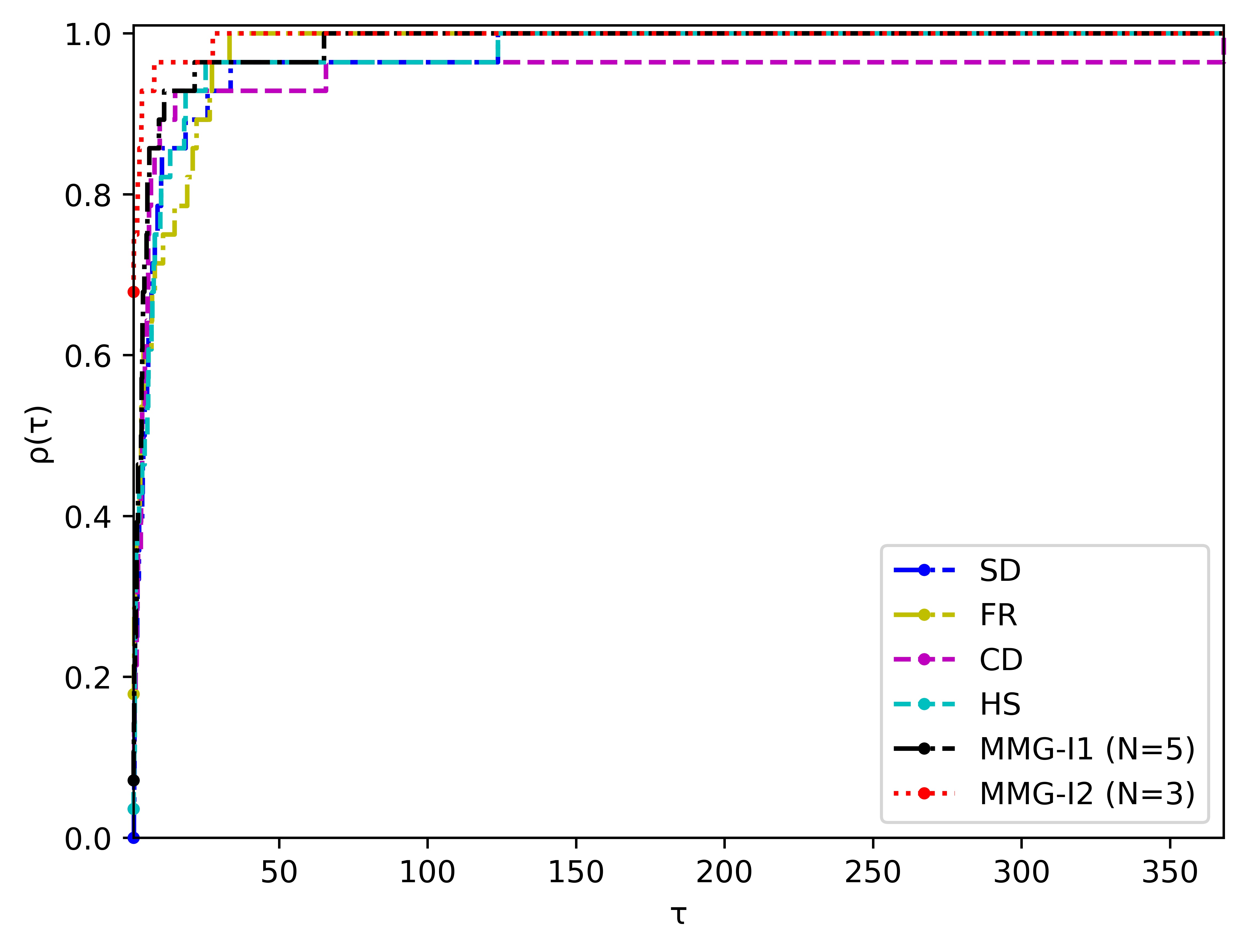}}
	\subfigure[Proformance profile on {\rm [0,30]}.]{
		\label{all_iter_tau=30}
		\includegraphics[width=0.45\textwidth]{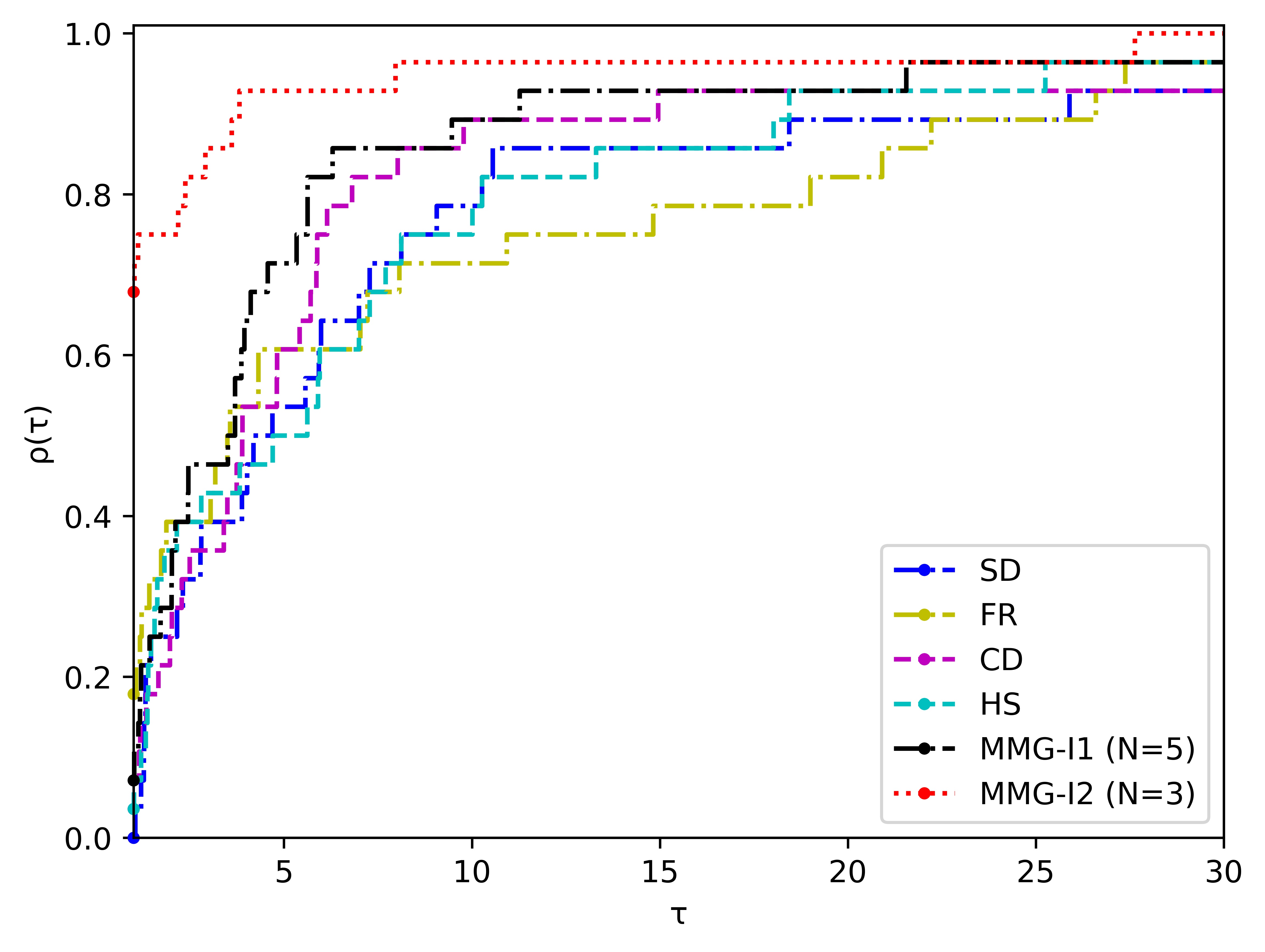}}
	\caption{Performance profiles comparing the number of iterations of all compared algorithms.}
	\label{all_iter}
\end{figure}

For the number of iterations, by Figure \ref{all_iter}, it is clear that MMG-I2 with $N=3$ has the most wins and that the probability that MMG-I2 with $N=3$ is the winner in view of iterations is about 67.86\%. Compared with SD, FR, CD and HS, MMG-I1 with $N=5$ draws our attention with its ability to solve problems, as presented by the performance profile for $\tau>2$ in Figure \ref{all_iter}(b).

\begin{figure}[H]
	\centering  
	\subfigure[Proformance profile {\rm [0,1200]}.]{
		\label{all_evalf_tau=None}
		\includegraphics[width=0.45\textwidth]{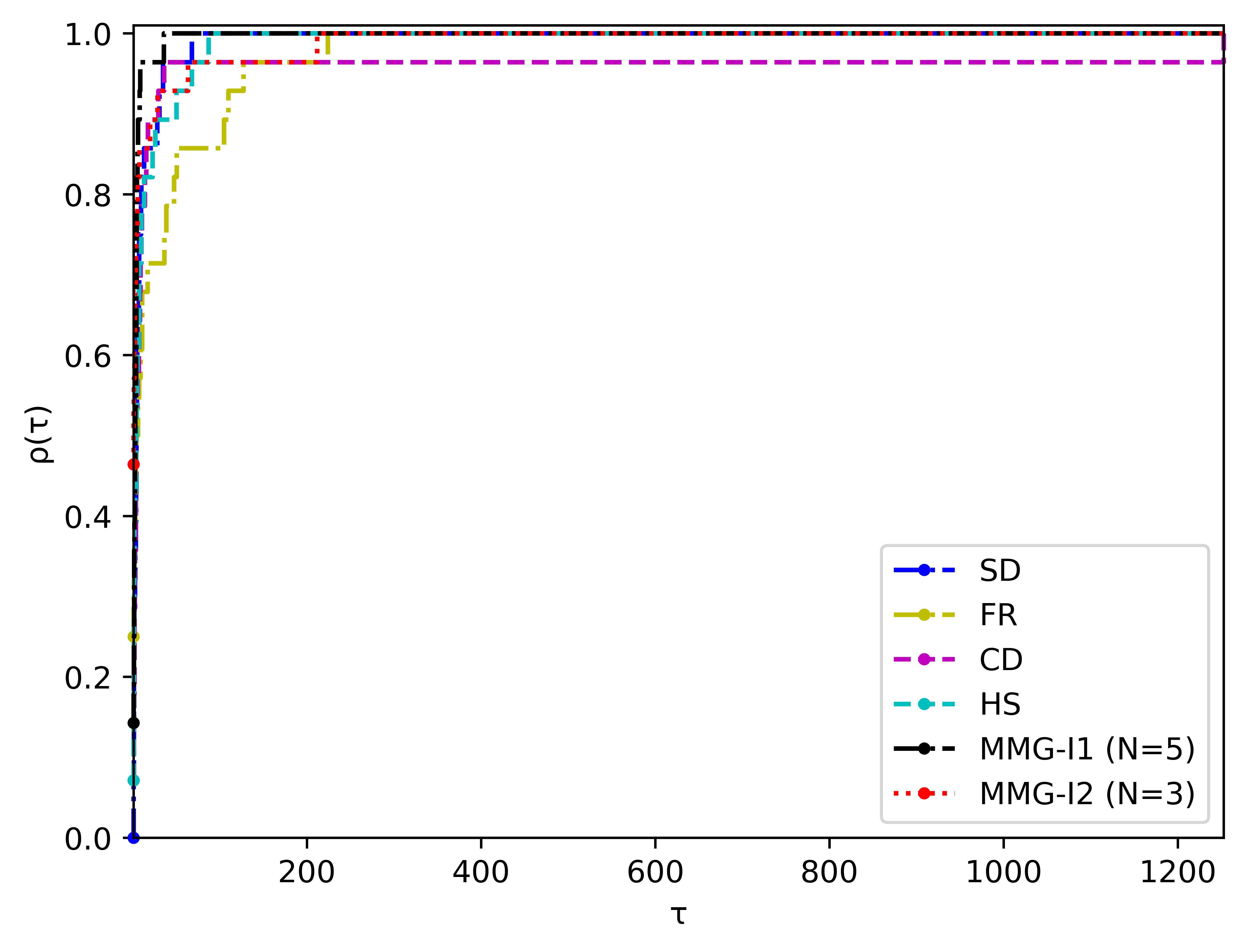}}
	\subfigure[Proformance profile on {\rm [0,150]}.]{
		\label{all_evalf_tau=200}
		\includegraphics[width=0.45\textwidth]{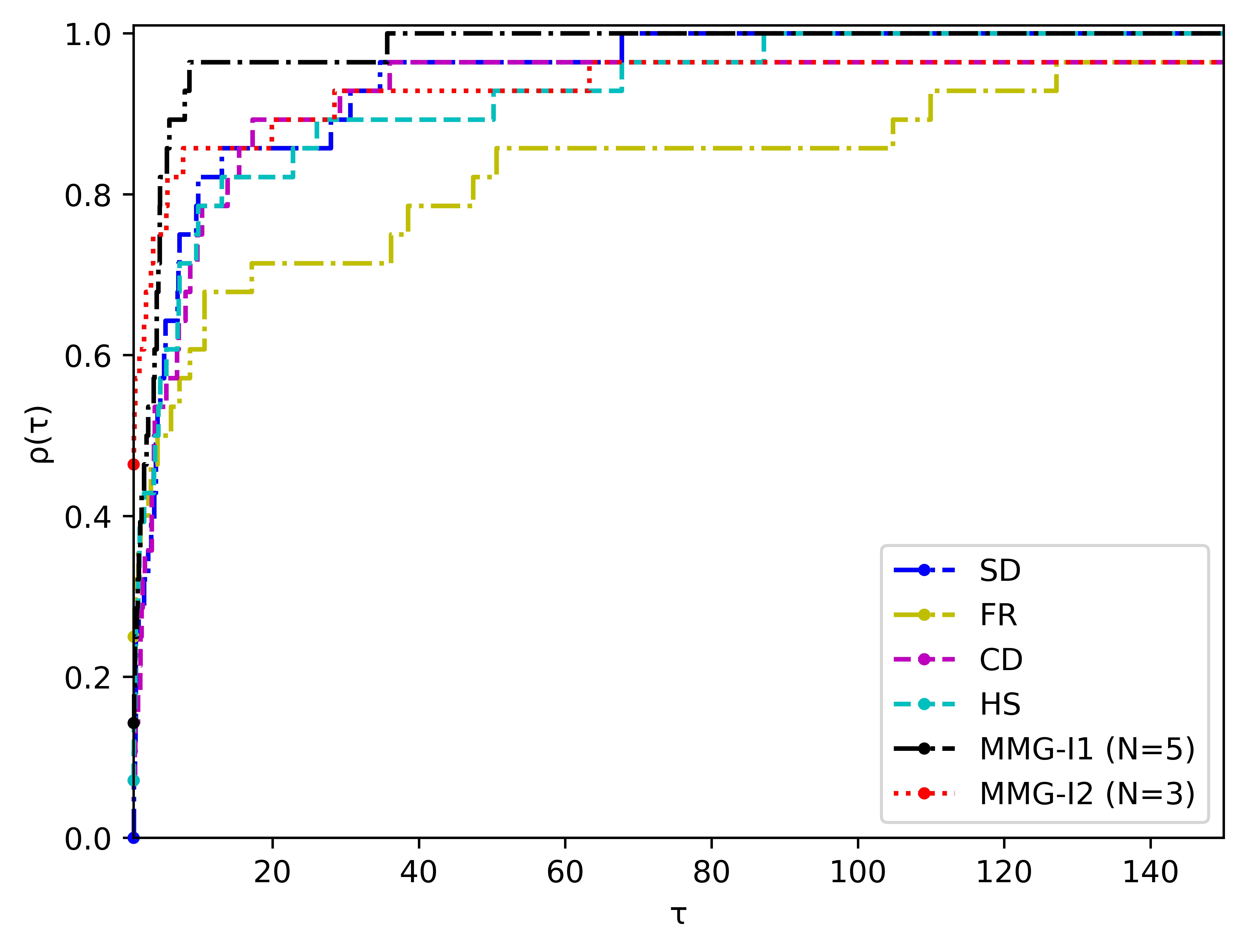}}
	\caption{Performance profiles comparing the number of function evaluations of all compared algorithms.}
	\label{mmgi2}
\end{figure}

For the number of function evaluations, from Figure \ref{mmgi2}, the probability that MMG-I2 with $N=3$ is the winner in function evaluations is about 46.64\%. MMG-I1 with $N=5$ shows its better performance when $\tau>5$. However, when $\tau>20$, the performance of MMG-I2 with $N=3$ is slightly inferior than that of other algorithms except FR in view of function evaluations. 

By Figure \ref{mmgi2}, a point of interest is that though MMG-I2 with $N=3$ consumes smaller iterations than the competitors, its performance about the number of function evaluations is slightly inferior than SD, CD, HS and MMG-I1 with $N=5$ in the interval $[25, 220]$. A reasonable explanation for this phenomenon is that it takes more evaluation to find a reasonable stepsize along the direction $d^{k}$ in the stepsize-I strategy. Therefore, the stepsize strategy also affects the practical behavior of MMG-I2. 

So far, we only care about the speed of convergence of our method. Next, we will focus on comparing the performance of these algorithms in the aspect of generating Pareto fronts (POFs). To this aim, we use the so-called \emph{purity} \cite{C_d2011} and \emph{spacing} \cite{S_f1995,G_m2008} metrics presented below. Let $F_{p,s}$ be a solution set found by solver $s\in\mathcal{S}$ on problem $p\in\mathcal{P}$. Suppose that $F_{p}$ is an approximation of the true POF for problem $p$, computed by first obtaining $\cup_{s\in\mathcal{S}}F_{ps}$ and then saving non-dominated points of this set.
\begin{itemize}
	\item \emph{Purity metric}. The ratio $\bar{\psi}_{ps}$ is defined by
	$$\bar{\psi}_{ps}=\frac{|F_{ps}\cap F_{p}|}{|F_{p}|}.$$
	In order to discuss the purity metric using the performance profiles, we let $\psi_{ps}=1/\bar{\psi}_{ps}$. As reported in \cite{C_d2011}, the algorithms are compared in pairs if we use the purity metric.
	\item \emph{Spacing metric}. The spacing metric is defined by
	$$Q_{p,s}=\left(\frac{1}{|{F}_{p,s}|-1}\sum_{l=1}^{|{F}_{p,s}|}(d^{*}-d_{l})^{2}\right)^{1/2},$$
	where 
	$$d_{l}=\min_{k}\left\{\sum_{i=1}^{m}|F_{i}(a_{l})-F_{i}(a_{k})|\right\},\quad a_{l},a_{k}\in{F}_{p,s},\quad l,k=1,2,\ldots,|{F}_{p,s}|,$$
	$d^{*}$ is the average value of all $d_{l}$. The lower values $Q_{p,s}$ indicate better performance.
\end{itemize}

Figure \ref{purity} depicts the performance profiles associated with the purity metric. As we can see, MMG-I1 with $N=5$ outperforms SD and HS, respectively. There is no significant difference between MMG-I1 with $N=5$ and CD.
MMG-I2 with $N=3$ outperforms than that of other compared algorithms. To sum up, the purity performance profiles show that the previous multi-step information makes the algorithm more effective. However, the superior performance of MMG-I2 with $N=3$ over MMG-I1 with $N=5$ may be surprising.

\begin{figure}[H]
	\centering  
	\subfigure{
		\label{sub.1}
		\includegraphics[width=0.3\textwidth]{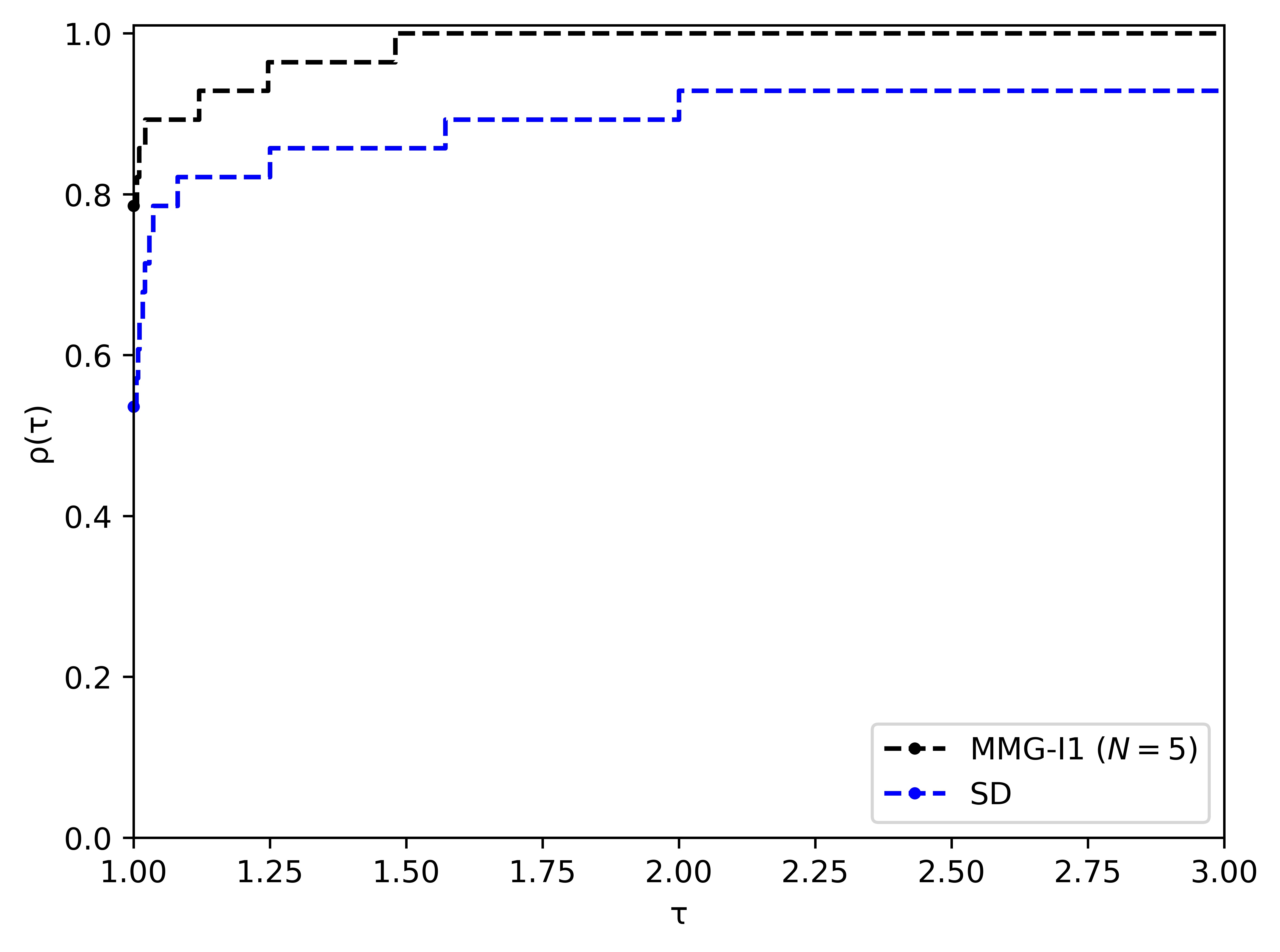}}
	\subfigure{
		\label{sub.2}
		\includegraphics[width=0.3\textwidth]{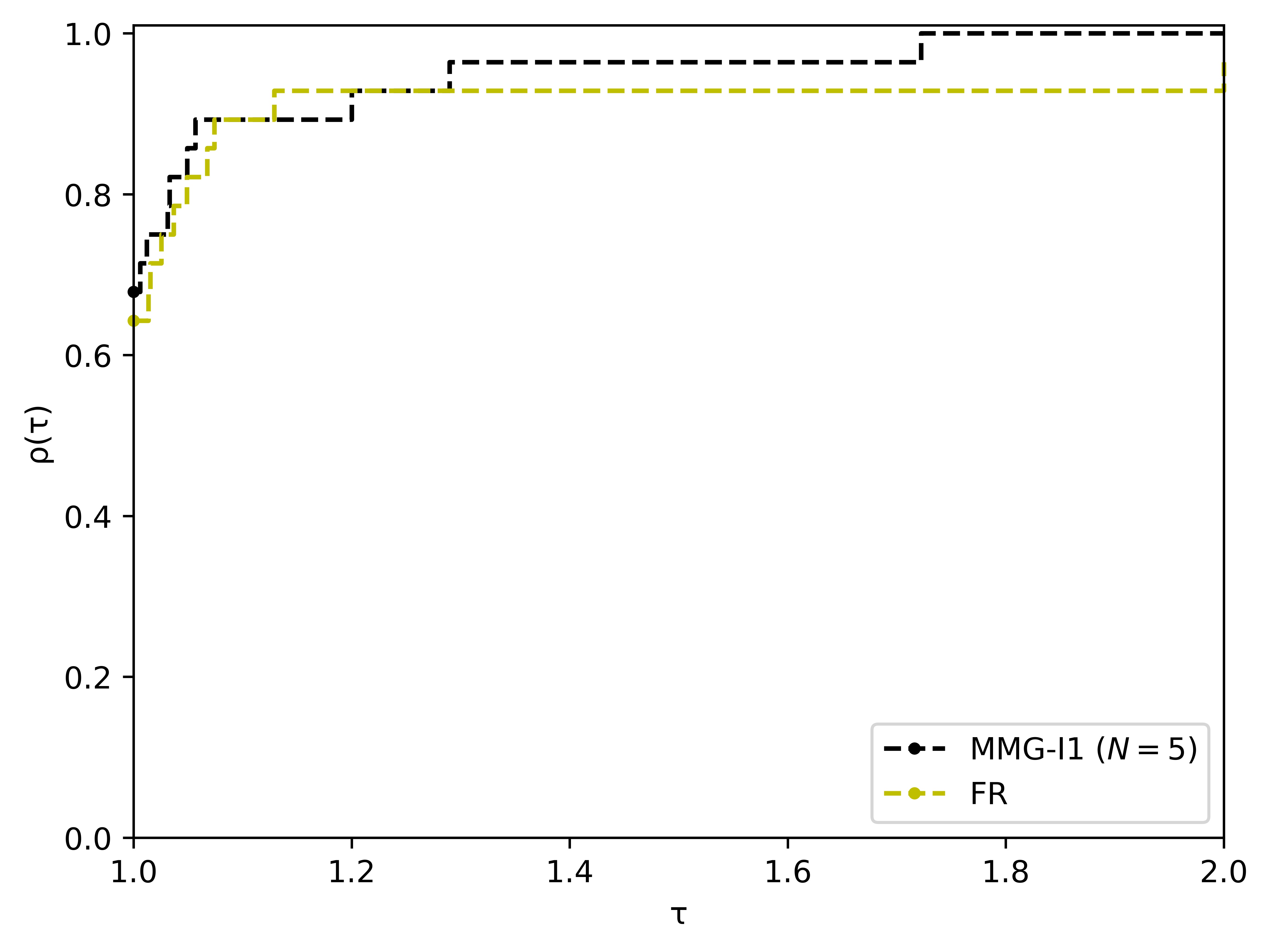}}
	\subfigure{
		\label{sub.3}
		\includegraphics[width=0.3\textwidth]{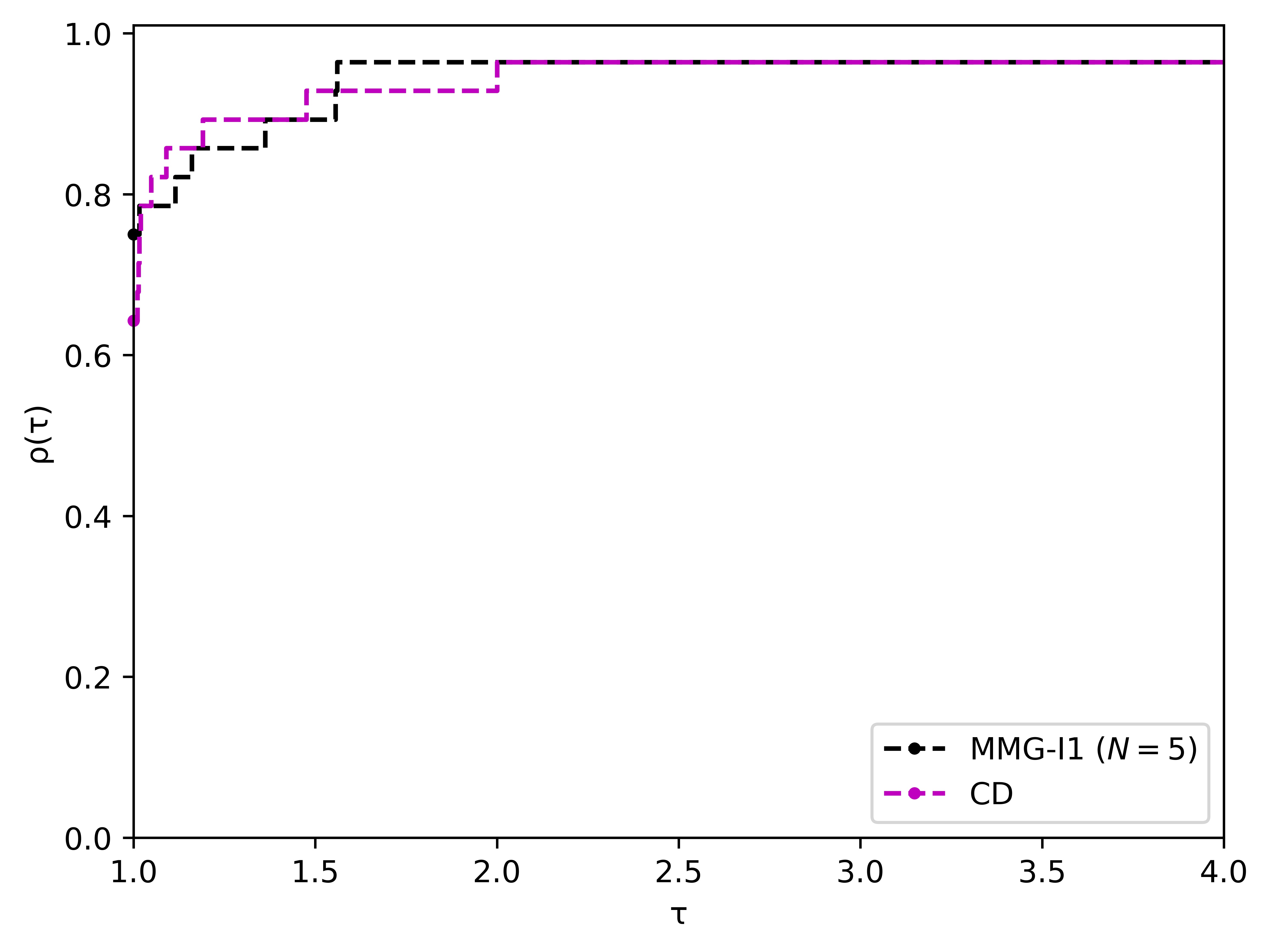}}
	\subfigure{
		\label{sub.4}
		\includegraphics[width=0.3\textwidth]{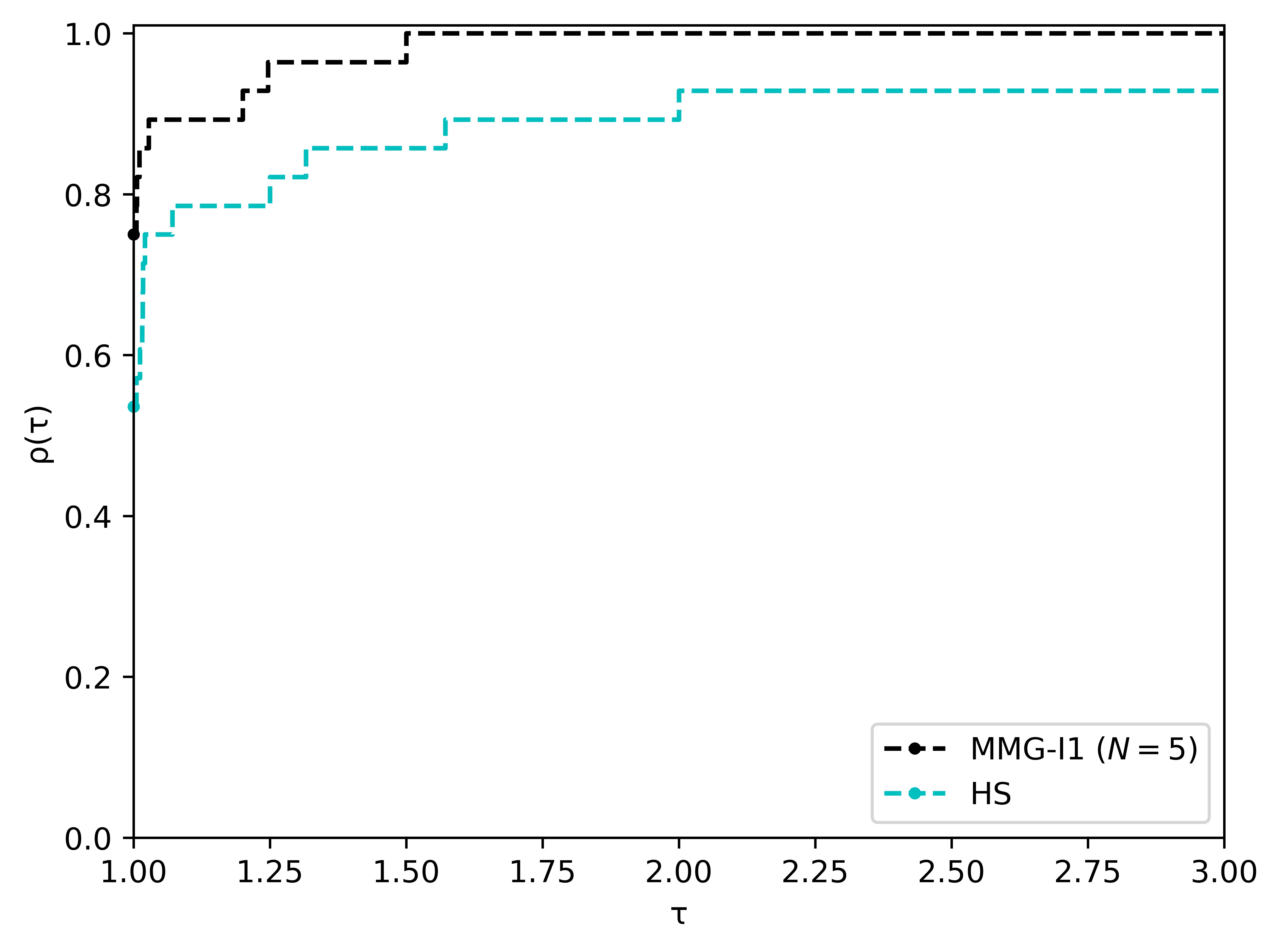}}
	\subfigure{
		\label{sub.5}
		\includegraphics[width=0.3\textwidth]{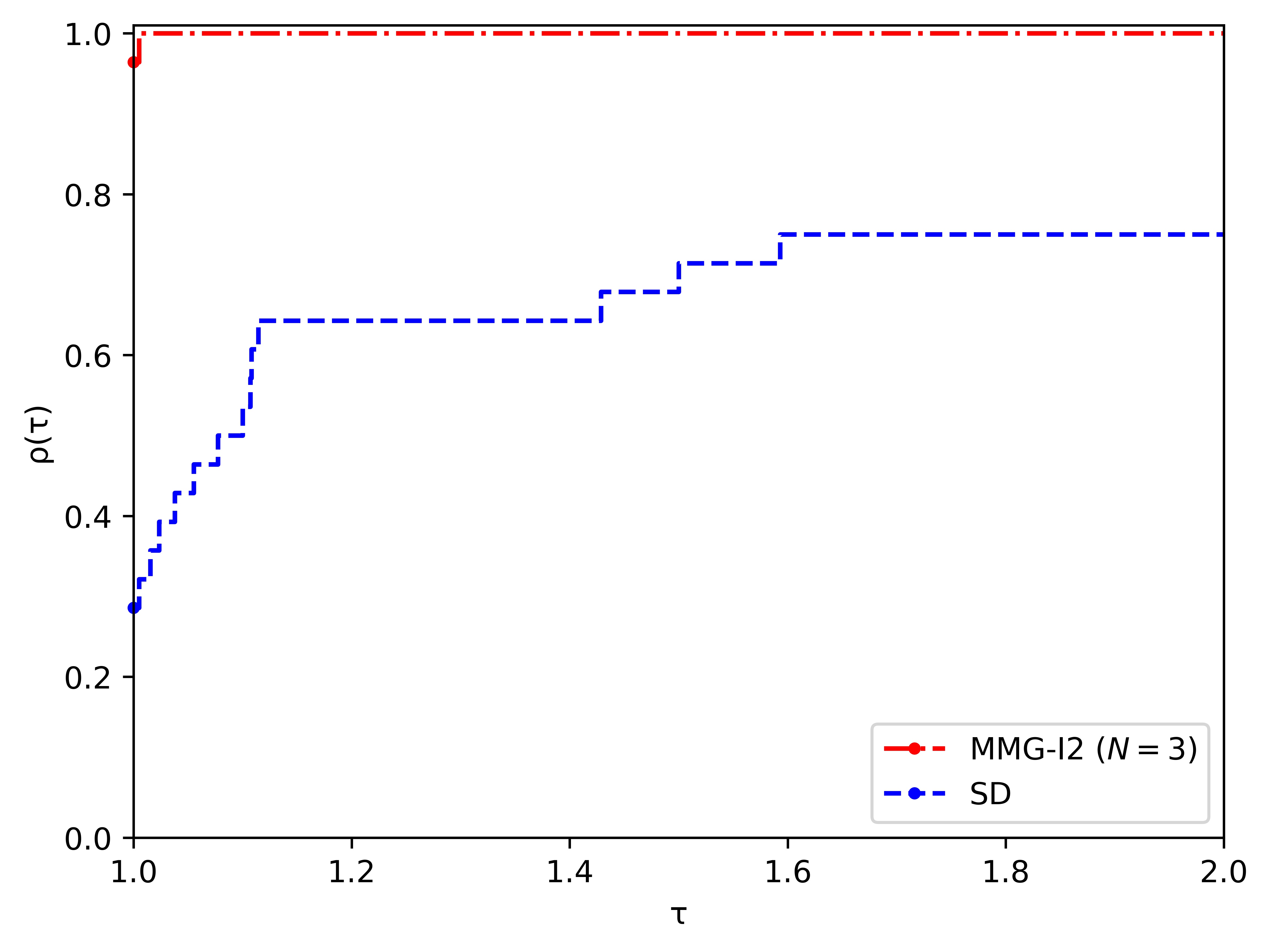}}
	\subfigure{
		\label{sub.6}
		\includegraphics[width=0.3\textwidth]{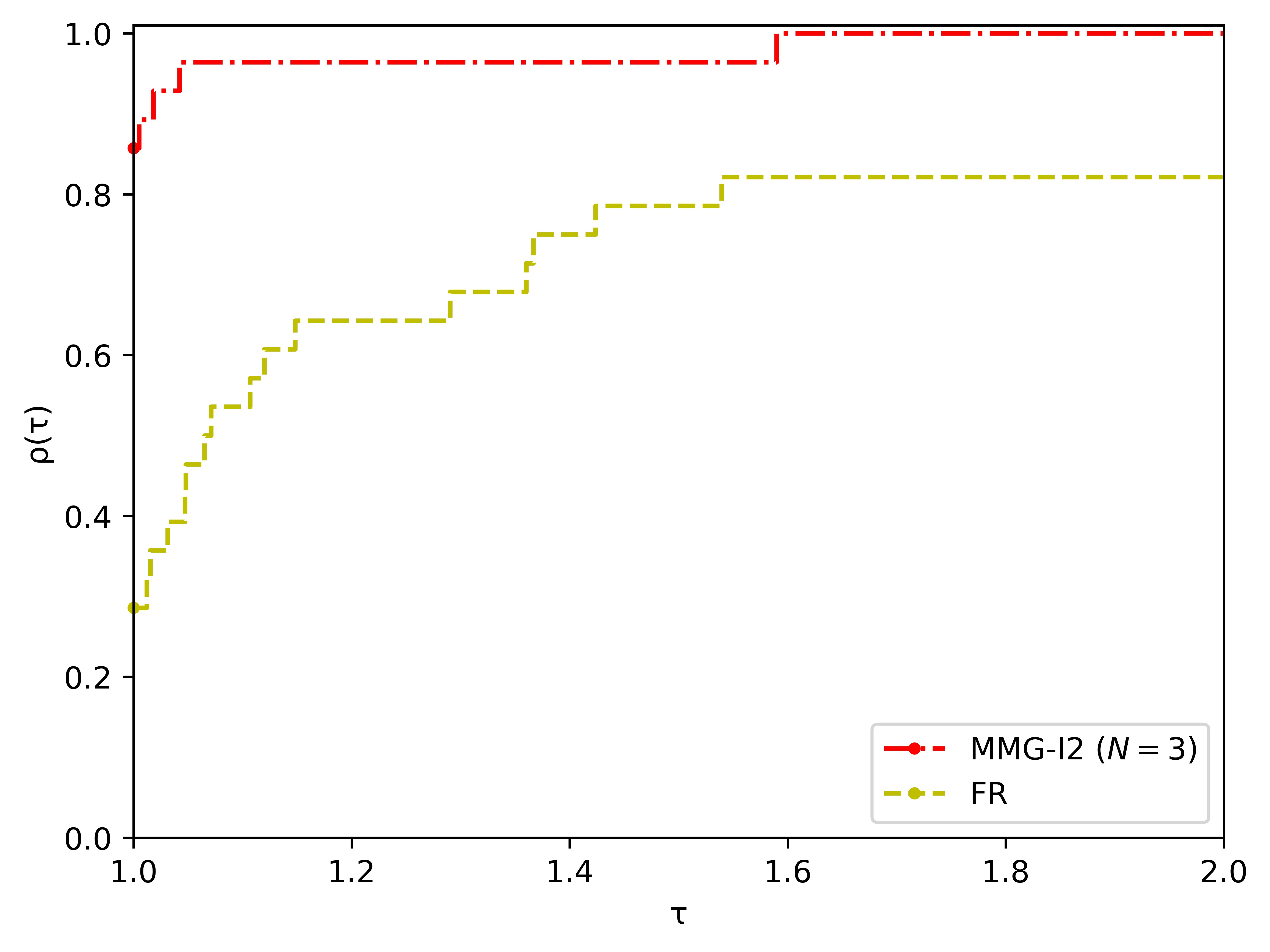}}
	\subfigure{
		\label{sub.7}
		\includegraphics[width=0.3\textwidth]{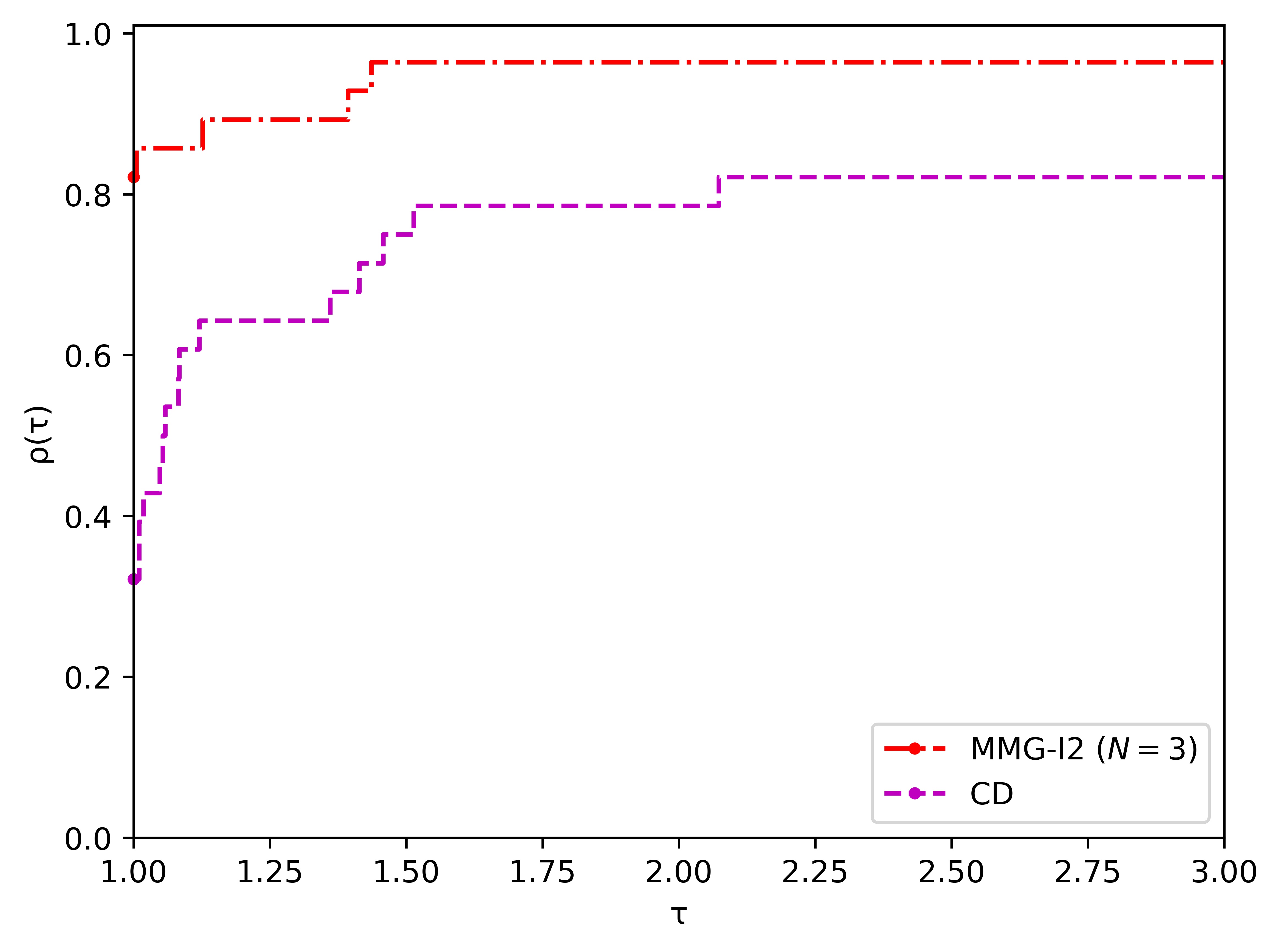}}
	\subfigure{
		\label{sub.8}
		\includegraphics[width=0.3\textwidth]{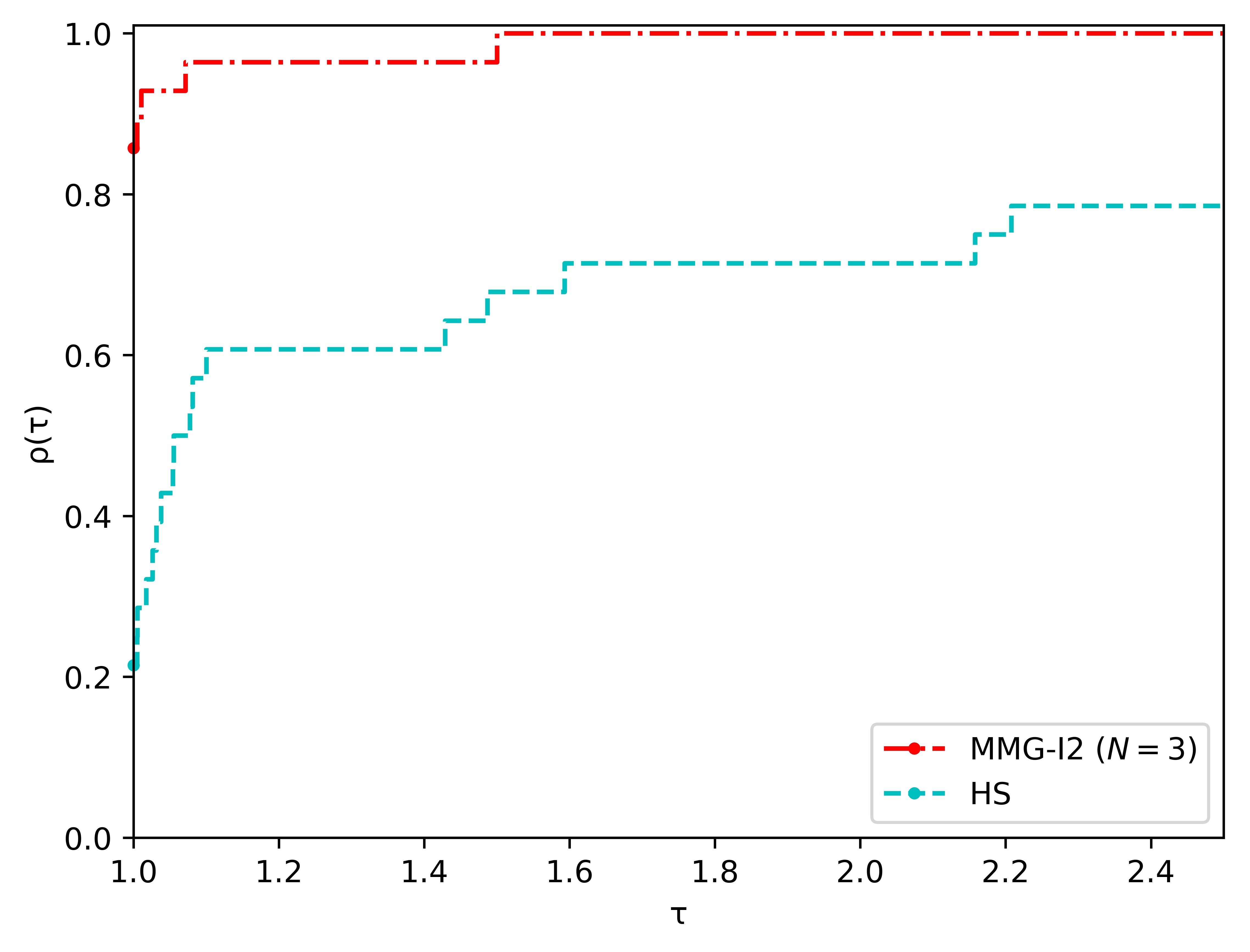}}
	\subfigure{
		\label{sub.9}
		\includegraphics[width=0.3\textwidth]{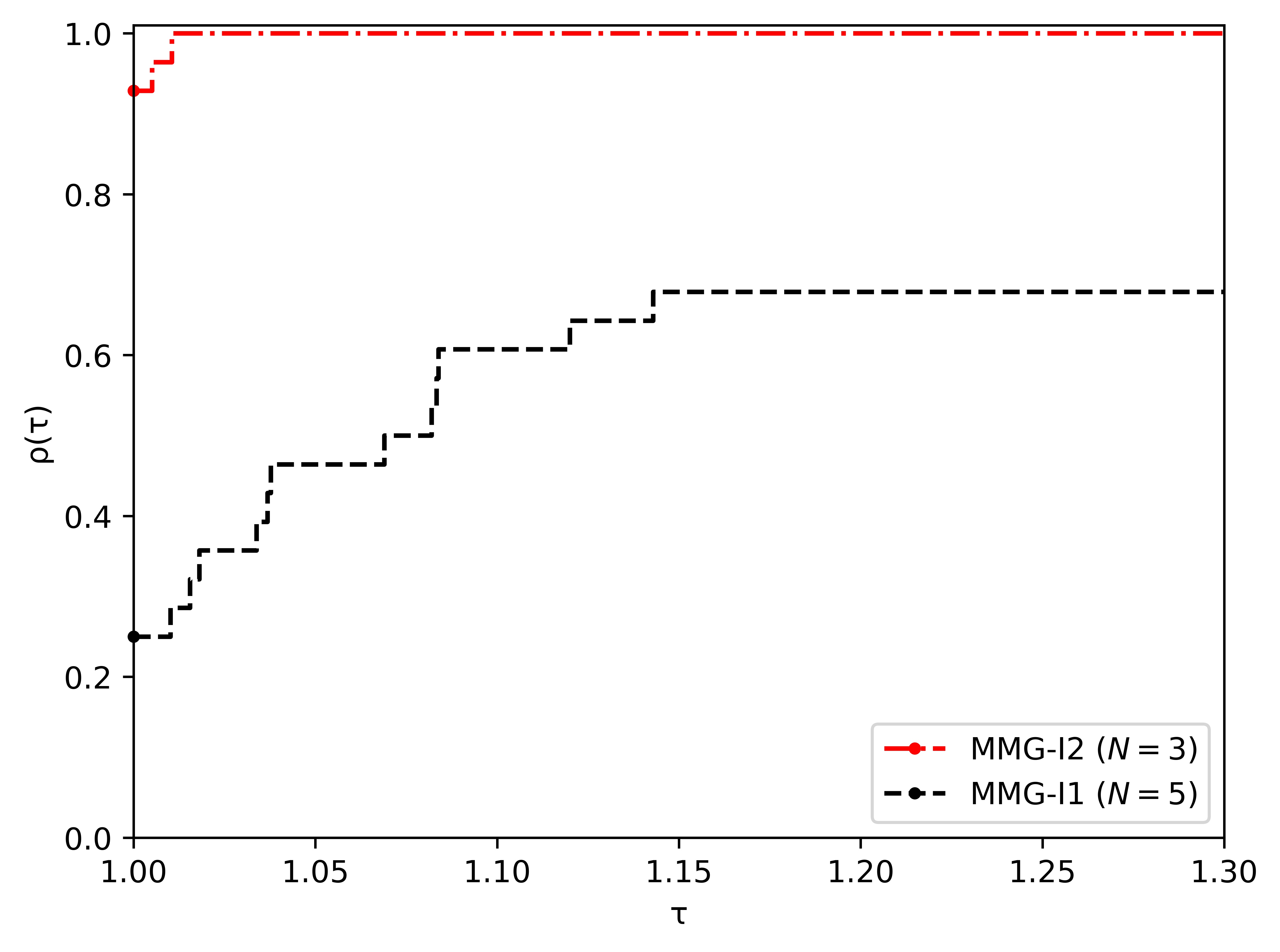}}
	\caption{Performance profile using the purity metric.}
	\label{purity}
\end{figure}

The values of the spacing metric are calculated and then are listed in Table \ref{spacing}. We highlight the best results for every problem with gray background. From Table \ref{spacing}, it follows that our method is slightly better than the others.

\begin{table}[H]\scriptsize
	\centering
	\setlength{\abovecaptionskip}{0cm}  
	\setlength{\belowcaptionskip}{-0.2cm} 
	\caption{The values of spacing metric of algorithms on each problems.}
	\begin{tabular}{lllllll}
		\hline
		\rowcolor{gray!20}& SD          & FR                                & CD                                & HS                                & MMG-I1 ($N=5$)                    & MMG-I2 ($N=3$)                    \\ \hline
		AP3               & 4.6466E+05  & 6.3104E+05                        & 4.8388E+05                        & 4.6525E+05                        & 5.3160E+05                        & \hl{3.2981E+05}  \\
		SK2               & 3.6158E-01  & 3.6789E-01                        & 3.6648E-01                        & 3.6158E-01                        & 3.9351E-01                        & \hl{2.5684E-01}  \\
		DD1$^{*}$         & 2.47076E+01 & 2.63280E+01                       & 2.44810E+01                       & 2.42144E+01                       & 2.38356E+01                       & \hl{1.61891E+01} \\
		DGO1              & 6.13259E-03 & 3.88022E-03                       & 3.79481E-03                       & 6.13259E-03                       & 3.64314E-03                       & \hl{3.53854E-03} \\
		DGO2              & 3.32536E-04 & 5.58581E-04                       & 3.61553E-04                       & 3.32536E-04                       & 3.67898E-04                       & \hl{1.20249E-04} \\
		TOI4$^{\dagger}$  & 3.93543E-03 & 8.77754E-03                       & 4.63074E-03                       & 6.37793E-03                       & \hl{3.18078E-03} & 3.85798E-03                       \\
		Far1              & 3.78246E-02 & \hl{1.08560E-02} & 2.84416E-02                       & 2.33181E-02                       & 4.42181E-02                       & 2.58139E-02                       \\
		BK1               & 5.67094E-03 & 5.72587E-03                       & 5.75176E-03                       & \hl{5.66572E-03}                       & 5.75848E-03                       & {5.74988E-03} \\
		LE1               & 3.86090E-02 & 4.43753E-02                       & 3.87667E-02                       & 4.65641E-02                       & 3.56940E-02                       & \hl{2.55921E-02} \\
		SLC2              & 3.97389E-04 & 4.05604E-04                       & \hl{3.90362E-04} & 3.97389E-04                       & 9.76401E-02                       & 1.40434E-03                       \\
		SD                & 6.38914E-03 & 6.38943E-03                       & 6.39005E-03                       & \hl{6.38914E-03} & 1.63398E-01                       & 1.65357E-02                       \\
		MOP2              & 4.32014E-02 & 4.73312E-02                       & 4.84541E-02                       & \hl{4.32014E-02} & 7.98203E-02                       & 8.00155E-02                       \\
		MOP3              & 2.15315E-01 & 2.15313E-01                       & 2.15315E-01                       & 2.15315E-01                       & 2.15315E-01                       & \hl{2.59174E-02} \\
		PNR               & 2.48868E-01 & 4.21765E+02                       & 2.20989E+00                       & 3.04759E-01                       & \hl{2.21242E-01} & 2.98312E-01                       \\
		VU1               & 8.51215E-03 & 9.86311E-03                       & \hl{4.35712E-03} & 8.22291E-03                       & 5.10587E-03                       & 6.02374E-03                       \\
		KW2               & 5.24998E+00 & 5.22418E+00                       & 5.28243E+00                       & 5.25269E+00                       & \hl{5.17345E+00} & 5.83021E+00                       \\
		MMR1$^{\star}$    & 4.77164E-01 & 4.91246E-01                       & 4.88122E-01                       & 4.92740E-01                       & 3.07840E-01                       & \hl{2.42224E-01} \\
		MMR3              & 5.83910E-03 & 6.08693E-03                       & \hl{5.74151E-03} & 5.82880E-03                       & 5.82848E-03                       & 5.78940E-03                       \\
		Lov3              & 2.36738E-02 & 1.82089E-02                       & 3.73455E-02                       & 6.05556E+01                       & 1.80085E-02                       & \hl{1.47742E-02} \\
		Lov4              & 6.64785E-03 & 6.64988E-03                       & 6.64882E-03                       & 6.64785E-03                       & \hl{6.64784E-03} & 6.64861E-03                       \\
		Lov6              & 1.68441E-02 & \hl{1.33253E-02} & 1.61420E-02                       & 1.61084E-02                       & 1.58396E-02                       & 1.63072E-02                       \\
		FF1               & 3.09501E-01 & 3.10884E-01                       & 3.48504E-01                       & 3.78128E-01                       & \hl{2.89823E-01} & 3.19258E-01                       \\
		FF1$^{\ddagger}$a & 2.91865E-02 & 2.91872E-02                       & \hl{2.91837E-02} & 2.91884E-02                       & 2.91865E-02                       & 2.91867E-02                       \\
		FF1$^{\ddagger}$b & 7.87607E-02 & 7.83368E-02                       & 5.31460E+03                       & 7.87607E-02                       & 7.87281E-02                       & \hl{7.74009E-02} \\
		JOS1a             & 2.18780E+00 & 5.30452E+00                       & 2.31023E+00                       & 2.13582E+00                       & \hl{2.09054E+00} & 2.69760E+00                       \\
		JOS1b             & 9.49827E-01 & 8.96478E-01                       & 1.20768E+00                       & 9.49827E-01                       & 8.78567E-01                       & \hl{3.25366E-01} \\
		JOS1c             & 2.54934E-01 & 2.49510E-01                       & \hl{2.36566E-01} & 2.63231E-01                       & 3.56407E-01                       & 6.18411E-01                       \\
		JOS1d             & 1.20539E-01 & 2.01532E-01                       & 4.05041E-01                       & 1.19843E-01                       & 1.17921E-01                       & \hl{1.14743E-01} \\ \hline
	\end{tabular}
	\label{spacing}
\end{table}

\section{Conclusions}
In this work, we have proposed a new descent method for solving unconstrained MOPs, which employs the past multi-step iterative information at every iteration. The developed search direction with suitable parameters in the proposed method has the sufficient descent property. Under mild assumptions, we derive the global convergence and convergence rates for our method. Numerical results are presented to demonstrate the efficiency of our method.

From the numerical experiments, it is worth mentioning that our method depends on the parameters $N$, $\gamma_{k}$ and $\psi_{kj}$, which directly determine the importance coefficient of previous information. How to select these parameters in our method deserves further study. As presented in Section 6.2, the stepsize strategy seems to be an additional factor in our method. Recently, some stepsize strategies, including Goldstein \cite{W_e2019}, Wolfe \cite{L_n2018}, nonmonotone \cite{M_n2019} line searches, have been proposed in multiobjective optimization. It would be interesting to study our method with these stepsize strategies in the future.

\end{document}